\documentclass[11pt]{amsart}
\usepackage{amsmath,amssymb,amsthm,amsfonts,stmaryrd,rotating,dsfont,bm}
\usepackage[a4paper]{geometry}
\usepackage[latin9]{inputenc} 
\usepackage[matrix,arrow,curve]{xy}
\DeclareMathOperator{\diag}{diag}
\newcommand{\otimesB}{\underset{B}{\otimes}}

\newcommand{\cg}{C^{*}(\Gamma)}
\newcommand{\cp}{\mathrm{c.p.}}
\newcommand{\maxtimes}{\hat{\otimes}}
\newcommand{\gmaxtimes}{\!\stackrel{\scriptscriptstyle \Gamma}{\maxtimes}\!}
\newcommand{\gbmaxtimes}{\!\underset{\scriptscriptstyle
    B}{\stackrel{\scriptscriptstyle \Gamma}{\maxtimes}}\!}

\newcommand{\Rn}{\mathcal{R}_{n}}
\newcommand{\barop}{\overline{\op}}

\newcommand{\bfF}{\mathbf{F}}
\newcommand{\bfC}{C^{*}(-)}
\newcommand{\bfG}{\mathbf{G}}
\newcommand{\bfU}{(-)_{*,*}}
\newcommand{\bgalg}{\ensuremath
  \mathbf{Alg}_{(B,\Gamma)}}
\newcommand{\bgsalg}{\ensuremath
  \mathbf{*\text{-}Alg}_{(B,\Gamma)}}
 \newcommand{\bgsalgu}{\ensuremath
  \mathbf{*\text{-}Alg}^{0}_{(B,\Gamma)}}
\newcommand{\cgsalgu}{\ensuremath
  \mathbf{*\text{-}Alg}^{0}_{(\complex,\Gamma)}}

\newcommand{\bgcalg}{\ensuremath
  \mathbf{C^{*}\text{-}Alg}_{(B,\Gamma)}}
\newcommand{\cgcalg}{\ensuremath
  \mathbf{C^{*}\text{-}Alg}_{(\complex,\Gamma)}}

\newcommand{\mal}{\tilde\boxtimes}

\newcommand{\Mn}{M_{n}}
\newcommand{\GLn}{\mathrm{GL}_{n}}
\newcommand{\Ao}{A^{B}_{\mathrm{o}}}
\newcommand{\Aoo}{A^{B_{0}}_{\mathrm{o}}}
\newcommand{\Au}{A^{B}_{\mathrm{u}}}
\newcommand{\Auo}{A^{B_{0}}_{\mathrm{u}}}
\newcommand{\Aup}{A^{B}_{\mathrm{u}'}}
\theoremstyle{definition} 

\swapnumbers
\newtheorem{remark}{Remark}[subsection]
\newtheorem{remarks}[remark]{Remarks}
\newtheorem{example}[remark]{Example}

\theoremstyle{plain}
\newtheorem{definition}[remark]{Definition}
\newtheorem{theorem}[remark]{Theorem}
\newtheorem{proposition}[remark]{Proposition}
\newtheorem{corollary}[remark]{Corollary}
\newtheorem{lemma}[remark]{Lemma}

\newtheorem*{assumption*}{Assumption}
\newtheorem*{definition*}{Definition}
\newtheorem*{theorem*}{Theorem}
\newcommand{\mfrak}[1]{\mathfrak{#1}}

\newcommand{\frakM}{\mfrak{M}}

\newcommand{\gtimes}{\stackrel{\Gamma}{\otimes}}

\newcommand{\complex}{\mathbb{C}}
\newcommand{\reals}{\mathbb{R}}
\newcommand{\integers}{\mathbb{Z}}
\newcommand{\naturals}{\mathbb{N}}

\DeclareMathOperator{\Id}{id}

\DeclareMathOperator{\End}{End}

\newcommand{\todot}{\tilde{\otimes}}

\newcommand{\Lt}[1]{\triangleleft}
\newcommand{\Rt}[1]{\triangleright}

\newcommand{\frei}{\,\cdot\,}

\newcommand{\SU}{\mathrm{SU}(2)}
\newcommand{\SUq}{\mathrm{SU}_{q}(2)}
\newcommand{\SUd}{\mathrm{SU}^{\mathrm{dyn}}_{Q}(2)}

\makeatother

\newcommand{\op}{\mathsf{op}}
\newcommand{\co}{\mathsf{co}}
\renewcommand{\top}{\mathsf{T}}
\newcommand{\rotxc}[1]{\begin{sideways}#1\end{sideways}}
\newcommand{\invert}[1]{\rotxc{\rotxc{#1}}}
\renewcommand{\bot}{\invert{$\scriptstyle\mathsf{T}$}}

\newcommand{\MC}{\frakM(\complex)}
\newcommand{\FR}{\mathcal{F}_{R}}
\newcommand{\SL}{\mathrm{SL(2)}}
\newcommand{\Hopf}{\mathbf{Hopf}}
\newcommand{\cHopf}{\mathbf{C^{*}}\text{-}\mathbf{Hopf}}
\newcommand{\btimes}{\underset{B}{\otimes}}

\newcommand{\ev}{\mathrm{ev}}

\newcommand{\cgalg}{\ensuremath
  \mathbf{Alg}_{(C,\Gamma)}}
\newcommand{\cgsalg}{\ensuremath
  \mathbf{Alg}^{*}_{(C,\Gamma)}}

\begin{document}



\title[Free dynamical quantum groups and $\SUd$]{Free dynamical quantum groups \\ and \\  the dynamical
  quantum group $\SUd$}

\author{Thomas Timmermann}
\thanks{Supported by the SFB 878
    ``Groups, geometry and actions''}
\address{University of Muenster \\ Einsteinstr.\ 62, 48149
  Muenster, Germany \\ timmermt@math.uni-muenster.de}
\date{\today}

\maketitle
\begin{abstract}
  We introduce dynamical analogues of the free orthogonal
  and free unitary quantum groups, which are no longer Hopf
  algebras but Hopf algebroids or quantum
  groupoids. These objects are constructed on the purely
  algebraic level and on the level of universal
  $C^{*}$-algebras. As an example, we recover the dynamical
  $\SUq$ studied by Koelink and Rosengren, and construct a
  refinement that includes several interesting limit cases.
\end{abstract}
\tableofcontents

\section{Introduction}

\label{section:introuction}

Dynamical quantum groups were introduced by Etingof and
Var\-chenko as an algebraic tool to study the quantum
dynamical Yang-Baxter equation appearing in statistical
mechanics
\cite{etingof:book,etingof:qdybe,etingof:exchange}. Roughly,
one can associate to every dynamical quantum group a
monoidal category of dynamical representations, and to every
solution $R$ of the dynamical Yang-Baxter equation a
dynamical quantum group $A_{R}$ with a specific
 dynamical representation $\pi$ such that
$R$ corresponds to a braiding on the monoidal category
generated by $\pi$.

In this article, we introduce two families of dynamical
quantum groups $\Ao(\nabla,F)$ and $\Au(\nabla,F)$ which are
natural generalizations of the free orthogonal and the free
unitary quantum groups introduced by Wang and van Daele
\cite{wang:universal,wang:thesis}. Roughly, these dynamical
quantum groups are universal with respect to the property
that they possess a corepresentation $v$ such that $F$
becomes a morphism of corepresentations from the inverse of
the transpose $v^{-\top}$ or from $(v^{-\top})^{-\top}$,
respectively, to $v$.

For a specific choice of $B,\nabla,F$, the free orthogonal
dynamical quantum group turns out to coincide with the
dynamical analogue of $\SUq$ that arises from a
trigonometric dynamical $R$-matrix and was studied by
Koelink and Rosengren \cite{koelink:su2}. We refine the
definition of this variant of $\SUq$ so that the resulting
global dynamical quantum group includes the classical $\SU$,
the non-dynamical $\SUq$ of Woronowicz \cite{woron:0}, the
dynamical $\SUq$ and further interesting limit cases which
can be recovered from the global object by suitable base
changes.

In the non-dynamical case, free orthogonal and free unitary
quantum groups are most conveniently constructed on the
level of universal $C^{*}$-algebras, where Woronowicz's
theory of compact matrix quantum groups applies
\cite{woron:1}.  We shall, however, start on the purely
algebraic level and then pass to the level of
universal $C^{*}$-algebras, where the main problem is to
identify a good definition of a dynamical quantum group.

These new classes of dynamical quantum groups give rise to
several interesting questions, for example, whether it is
possible to obtain a classification similar as in
\cite{wang:classify}, to determine their categories of
representations as in \cite{banica:orthogonal} and
\cite{banica:unitary}, or to relate their representation
theory to special functions as it was done in
\cite{koelink:su2} in the special case of $\SUq$.

Let us now describe the organization and contents of this
article in some more detail.

The first part of this article (\S\ref{section:algebra}) is
devoted to the purely algebraic setting.  

We start with a summary on dynamical quantum groups
(\S\ref{section:bg}). Roughly, these objects can be regarded
as Hopf algebras, that is, as algebras $A$ equipped with a
comultiplication $\Delta$, counit $\epsilon$ and antipode
$S$, where the field of scalars has been replaced by a
commutative algebra $B$ equipped with an action of a group
$\Gamma$. The comultiplication $\Delta$ does not take values
in the ordinary tensor product $A\otimes A$, but in a
product $A\todot A$ that takes $B$ and $\Gamma$ into
account, and the counit takes values in the crossed product
algebra $B\rtimes \Gamma$ which is the unit for the
product $-\todot-$.  If $B$ is trivial, however, these dynamical
quantum groups are just $\Gamma$-graded Hopf algebras
(\S\ref{section:trivial}). In general, we shall use the term
$(B,\Gamma)$-Hopf algebroid instead of dynamical quantum
group to be more precise.

The free orthogonal and unitary dynamical quantum
groups are  defined as follows.  Let $B$ be a unital,
commutative algebra with a left action of a group $\Gamma$,
let $\nabla=(\gamma_{1},\ldots,\gamma_{n})$ be an $n$-tuple
in $\Gamma$ and let $F \in \GLn(B)$ such that $F_{ij} = 0$
whenever $\gamma_{i} \neq \gamma_{j}^{-1}$.
\begin{definition*} \label{definition:intro-ao}  The \emph{free orthogonal dynamical
    quantum}  $\Ao(\nabla,F)$ is the
  universal algebra  with a homomorphism $r
  \times s\colon B\otimes B\to \Ao(\nabla,F)$ and a $v\in
  \GLn(\Ao(\nabla,F))$ satisfying
    \begin{itemize}
    \item[(a)] $v_{ij}r(b)s(b') = r(\gamma_{i}(b))
      s(\gamma_{j}(b'))v_{ij}$ for all $b,b'\in B$ and $i,j
      \in \{1,\ldots,n\}$,
    \item[(b)] $r_{n}(\hat F)v^{-\top} =vs_{n}(F)$, where
      $v^{-\top}$ denotes the transpose of $v^{-1}$ and $
      \hat F = (\gamma_{i}(F_{ij}))_{i,j}$.
    \end{itemize}
\end{definition*}
\begin{theorem*} \label{theorem:intro-ao-hopf}
  $\Ao(\nabla,F)$ can be equipped with the structure of a
  $(B,\Gamma)$-Hopf algebroid such that $\Delta(v_{ij}) =
  \sum_{k} v_{ik} \todot v_{kj}$, $ \epsilon(v_{ij}) =
  \delta_{i,j}\gamma_{i}$, and $ S(v_{ij}) = (v^{-1})_{ij}$
  for all $i,j$.
\end{theorem*}

Assume now that $B$ is equipped with an involution and let
$F\in \GLn(B)$ such that $F^{*}=F$ and $F_{ij}=0$ whenever
$\gamma_{i} \neq \gamma_{j}$.
\begin{definition*} The \emph{free unitary dynamical
    quantum} $\Au(\nabla,F)$ is the universal $*$-algebra
  with a homomorphism $r \times s\colon B\otimes B\to
  \Au(\nabla,F)$ and a unitary $v\in \GLn(\Au(\nabla,F))$
  satisfying the condition (a) above and (c) $\bar v$ is
  invertible and $r_{n}(\hat F)\bar v^{-\top} =vs_{n}(F)$.
\end{definition*}
\begin{theorem*}
  $\Au(\nabla,F)$ can be equipped with the structure of a
  $(B,\Gamma)$-Hopf $*$-algebroid such that $\Delta(v_{ij}) =
  \sum_{k} v_{ik} \todot v_{kj}$, $ \epsilon(v_{ij}) =
  \delta_{i,j}\gamma_{i}$, $ S(v_{ij}) = (v^{-1})_{ij}$ for
  all $i,j$.
\end{theorem*}
The formulas for $\Delta(v_{ij})$ and $\epsilon(v_{ij})$
above imply that the matrices $v$ above are
corepresentations of $\Ao(\nabla,F)$ and $\Au(\nabla,F)$,
respectively, and the conditions (b) and (c) assert that $F$
is an intertwiner from $v^{-\top}$ or $\bar v^{-\top}$,
respectively, to $v$. Such intertwiner relations admit
plenty functorial transformations which are studied
systematically in \S\ref{section:rn}, and yield short proofs
of the results above in \S\ref{section:ao}.  There, we also
consider involutions on certain quotients $\Ao(\nabla,F,G)$
of $\Ao(\nabla,F)$ which are parameterized by an
additional matrix $G\in \GLn(B)$.

Interestingly, the square of the antipode on the dynamical
quantum groups $\Ao(\nabla,F)$ and $\Au(\nabla,F)$ can be
described in terms of a natural family of characters
$(\theta^{(k)})_{k}$ which, like the counit $\epsilon$, take
values in $B\rtimes \Gamma$.  This family is an analogue of 
Woronowicz's fundamental family of characters on a compact
quantum group.

As a main example of the constructions above, we recover the
dynamical quantum group $\mathcal{F}_{R}(\mathrm{SU}(2))$ of
Koelink and Rosengren \cite{koelink:su2} associated to a
deformation parameter $q \neq 1$ as the free orthogonal
dynamical quantum group $\Ao(\nabla,F,G)$, where $B$ is the
meromorphic functions on the plane, $\Gamma=\integers$
acting by shifts, $\nabla=(1,-1)$ and $F=
\begin{pmatrix}
  0 & 1 \\ \tilde f & 0
\end{pmatrix}
$, where $\tilde f$ is the meromorphic function $\lambda
\mapsto q^{-1}(q^{2\lambda}-q^{-2})/(q^{2\lambda}-1)$, and $G=
\begin{pmatrix}
  0 & -1 \\ q^{-1} & 0
\end{pmatrix}
$. In \S\ref{section:sud}, we show how this example can be
refined such that the resulting dynamical quantum group
$\Ao(\nabla,F,G)$ includes $\mathcal{F}_{R}(\mathrm{SU}(2))$
and, simultaneously, a number of interesting limit cases
which can be recovered from the global object by suitable
base changes.

The second part of this article (\S\ref{section:universal})
extends the definition of dynamical quantum groups to the
level of universal $C^{*}$-algebras. Here, $B$ is assumed to
be a unital, commutative $C^{*}$-algebra and $\Gamma$ acts
via automorphisms. The main tasks is to find a
$C^{*}$-algebraic analogue of the product $-\todot -$ that
describes the target of the comultiplication. As in the
algebraic setting, we construct this product in two steps, by first
forming a cotensor product with respect to the Hopf
$C^{*}$-algebra $C^{*}(\Gamma)$ naturally associated to the
group $\Gamma$ (\S\ref{section:c}), and then taking a
quotient with respect to $B$ (\S\ref{section:cb}).  Given
the monoidal product, all definitions carry over from the
algebraic setting to the setting of universal
$C^{*}$-algebras easily (\S\ref{section:free-c}).

\section{The purely algebraic level}

\label{section:algebra}

Throughout this section, we assume all algebras and
homomorphisms to be unital over a fixed common ground field,
and $B$ to be a commutative algebra equipped with a left
action of a group $\Gamma$.

\subsection{Preliminaries on dynamical quantum groups}

\label{section:bg}

This subsection summarizes the basics of dynamical quantum
groups used in this article. We introduce the monoidal
category of $(B,\Gamma)$-algebras,  then
define $(B,\Gamma)$-Hopf algebroids, and finally consider
base changes and
the setting of $*$-algebras. Except for the base change,
most of this material is contained in \cite{etingof:qdybe}
and \cite{koelink:su2} in slightly different guise.  We omit
all proofs because they are straightforward.

Let $B^{\ev}=B\otimes B$.  A \emph{$B^{\ev}$-algebra} is an
algebra with a homomorphism $r\times s\colon B^{\ev}\to A$,
or equivalently, with homomorphisms $r_{A}=r,s_{A}=s\colon
B\to A$ whose images commute.  A morphism of
$B^{\ev}$-algebras is a $B^{\ev}$-linear homomorphism.
Write $\Gamma^{\ev}=\Gamma\times \Gamma$ and let
$e\in\Gamma$ be the unit.  Given a $\Gamma^{\ev}$-graded
algebra $A$, we write $\partial_{a} =
(\partial^{r}_{a},\partial^{s}_{a}) = (\gamma,\gamma')$
whenever $a\in A_{\gamma,\gamma'}$.
\begin{definition} \label{definition:bg-algebra} A
  \emph{$(B,\Gamma)$-algebra} is a $\Gamma^{\ev}$-graded
  $B^{\ev}$-algebra such that $(r\times s)(B^{\ev})
  \subseteq A_{e,e}$ and $ar(b)= r(\partial^{r}_{a}(b))a$,
  $as(b)=s(\partial^{s}_{a}(b)) a$ for all $b \in B$,
  $a\in A$.  A \emph{morphism} of $(B,\Gamma)$-algebras is a
  morphism of $\Gamma^{\ev}$-graded $B^{\ev}$-algebras.  We
  denote  by
  $\bgalg$ the category of all $(B,\Gamma)$-algebras.
\end{definition}
\begin{example} \label{example:bg-unit} Denote by $B\rtimes
  \Gamma$ the crossed product, that is, the universal
  algebra containing $B$ and $\Gamma$ such that $e=1_{B}$
  and $b\gamma \cdot b'\gamma' = b\gamma(b') \gamma\gamma'$
  for all $b,b'\in B$, $\gamma,\gamma' \in \Gamma$. This is
  a $(B,\Gamma)$-algebra, where
  $\partial_{b\gamma}=(\gamma,\gamma)$ and
  $r(b)=s(b)=b$ for all $b\in B$, $\gamma\in \Gamma$.
\end{example}
The category of all $(B,\Gamma)$-algebras can be equipped with a
 monoidal structure \cite{maclane} as follows. Let $A$ and $C$ be
$(B,\Gamma)$-algebras. Then the subalgebra
\begin{align*}
  A \gtimes C &:= \sum_{\gamma,\gamma',\gamma'' \in
    \Gamma} A_{\gamma,\gamma'} \otimes
  C_{\gamma',\gamma''} \subset A \otimes C
\end{align*}
is a $(B,\Gamma)$-algebra, where 
$\partial_{a\otimes c} = (\partial^{r}_{a},\partial^{s}_{c})$
 for all $a\in A$, $c\in C$ and $(r\times s)(b\otimes b') =
 r_{A}(b) \otimes s_{C}(b')$ for all $b,b'\in B$.
Let $I\subseteq A\gtimes C$ be the ideal generated by
 $\{s_{A}(b)\otimes 1 - 1\otimes r_{C}(b) : b\in
 B\}$. Then $A \todot C:=A\gtimes C/I$ is a
 $(B,\Gamma)$-algebra again, called the \emph{fiber product}
 of $A$ and $C$.  Write $a\todot c$ for the image of an
 element $a\otimes c$ in $A\todot C$.

The product $(A,C) \mapsto A\todot C$ is functorial,
associative and unital in the following sense.
\begin{lemma} \label{lemma:bg-monoidal}
  \begin{enumerate}
  \item For all morphisms of $(B,\Gamma)$-algebras
    $\pi^{1} \colon A^{1}\to C^{1}$, $\pi^{2}\colon
    A^{2}\to C^{2}$, there exists a morphism $\pi^{1}
    \tilde\otimes \pi^{2} \colon A^{1} \tilde\otimes A^{2}
    \to C^{1} \tilde\otimes C^{2}$, $a_{1} \todot a_{2}
    \mapsto \pi^{1}(a_{1}) \todot \pi^{2}(a_{2})$.
  \item For all $(B,\Gamma)$-alge\-bras $A,C,D$, there is an
    isomorphism $(A \tilde\otimes C) \tilde\otimes D \to A
    \tilde\otimes (C \tilde\otimes D)$, $(a \todot c) \todot
    d \mapsto a \todot (c \todot d)$.
  \item For each $(B,\Gamma)$-algebra $A$, there exist
    isomorphisms $(B\rtimes \Gamma) \tilde\otimes A \to A$
    and $A \tilde\otimes (B\rtimes \Gamma) \to A$, given by
    $b\gamma \todot a \mapsto r(b)a$ and $a \todot b\gamma
    \mapsto s(b)a$, respectively.
  \end{enumerate}
\end{lemma}
Of course, the isomorphisms above are compatible in a natural sense.
\begin{remark} \label{remark:takeuchi}
  The product $-\todot -$ is related to the left and right
  Takeuchi products $-{_{B}\times} -$ and $-\times_{B}-$ as
  follows.  Given a $B^{\ev}$-algebra $A$, we write
  ${_{\bullet}A}$ or $A_{\bullet}$ when we regard $A$ as a
  $B$-bimodule via $b \cdot a \cdot b':=r(b)s(b')a$ or $b
  \cdot a \cdot b':=ar(b)s(b')$, respectively.  Then the
  left and  right Takeuchi products of $B^{\ev}$-algebras
  $A$ and $C$ are the $B^{\ev}$-algebras
  \begin{align*}
    A {} {_{B}\times} C &:= \left\{\sum_{i} a_{i} \otimesB
      c_{i} \in  {_{\bullet} A \otimesB
        {_{\bullet}C}} \,\middle|\, \forall b\in B: 
\sum_{i} a_{i}s_{A}(b) \otimesB
      c_{i} = \sum_{i} a_{i} \otimesB
      c_{i}r_{C}(b)
      \right\}, \\
    A \times_{B} C &:=  \left\{\sum_{i} a_{i} \otimesB
      c_{i} \in  {A_{\bullet} \otimesB C_{\bullet}}\,\middle|\, \forall b\in B: 
\sum_{i} s_{A}(b)a_{i} \otimesB
      c_{i} = \sum_{i} a_{i} \otimesB
r_{C}(b)      c_{i}
      \right\},
  \end{align*}
  where the multiplication is defined factorwise and the
  embedding of $B^{\ev}$ is given by $b\otimes b'\mapsto
  r_{A}(b) \otimesB s_{C}(b')$.  The assignments $(A,C)
  \mapsto A {}{_{B}\times} C$ and $(A,C) \mapsto A
  {\times_{B}} C$ extend to bifunctors on the category of
  $B^{\ev}$-algebras and turn it into a lax monoidal
  category \cite{day:quantum-cat}.  The obvious forgetful
  functor $U$ from $(B,\Gamma)$-algebras to
  $B^{\ev}$-algebras is compatible with these products in
  the sense that for every pair of $(B,\Gamma)$-algebras
  $A,C$, the inclusion $A \gtimes C \hookrightarrow A\otimes
  C$ factorizes to inclusions of $A\todot C$ into $ A
  {}{_{B}\times} C$ and $A \times_{B} C$, yielding natural
  transformations from $U(-\todot -)$ to $U(-) {_{B}\times}
  U(-)$ and $U(-)\times_{B} U(-)$, respectively.
\end{remark}

Briefly, a $(B,\Gamma)$-Hopf algebroid is a coalgebra in
$\bgalg$ equipped with an antipode. To make this definition
precise, we need two involutions on $\bgalg$. Given an
algebra $A$, we denote by $A^{\op}$ its opposite, that is,
the same vector space with reversed multiplication.
  \begin{lemma}
    There exist automorphisms $(-)^{\op}$ and $(-)^{\co}$ of
    $\bgalg$ such that for each $(B,\Gamma)$-algebra $A$ and
    each morphism $\phi \colon A\to C$, we have $A^{\co}=A$
    as an algebra and
    \begin{align*}
      (A^{\op})_{\gamma,\gamma'} &=
      A_{\gamma^{-1},\gamma'{}^{-1}} \text{ for all }
      \gamma,\gamma'\in \Gamma, & r_{A^{\op}} &= r_{A}, &
      s_{A^{\op}} &= s_{A}, & \phi^{\op} &= \phi,
      \\
      (A^{\co})_{\gamma,\gamma'} &= A_{\gamma',\gamma}
      \text{ for all } \gamma,\gamma'\in \Gamma, &
      r_{A^{\co}} &= s_{A}, & s_{A^{\co}} &= r_{A}, &
      \phi^{\co} &= \phi.
    \end{align*}
    Furthermore, $(-)^{\op} \circ (-)^{\op} = \Id$,
    $(-)^{\co} \circ (-)^{\co} = \Id$, $(-)^{\op} \circ
    (-)^{\co} = (-)^{\co} \circ (-)^{\op}$.
  \end{lemma}
\begin{remark}
  The automorphisms above are compatible with the monoidal
  structure as follows. Given $(B,\Gamma)$-algebras $A,C$,
  there exist isomorphisms $(A \todot C)^{\op} \to (A^{\op}
  \todot C)^{\op}$ and $ (A \todot C)^{\co} \to C^{\co}
  \todot A^{\co}$ given by $a \todot c \mapsto a \todot c$
  and $a \todot c \mapsto c \todot a$, respectively.
  Moreover, $(B\rtimes \Gamma)^{\co} = B\rtimes \Gamma$ and
  there exists an isomorphism $S^{B\rtimes \Gamma} \colon
  B\rtimes \Gamma \to (B\rtimes \Gamma)^{\op}$, $b\gamma
  \mapsto \gamma^{-1} b$, and all of these isomorphisms and
  the isomorphisms in Lemma \ref{lemma:bg-monoidal} are
  compatible in a natural sense.
\end{remark}

\begin{definition} \label{definition:bg-hopf} A
  \emph{$(B,\Gamma)$-Hopf algebroid} is a
  $(B,\Gamma)$-algebra $A$ equipped with morphisms $\Delta
  \colon A\to A\todot A$, $\epsilon \colon A \to B\rtimes
  \Gamma$, and $S \colon A \to A^{\co,\op}$ such that the
  diagrams below commute,
  \begin{gather*}
    \xymatrix@R=15pt@C=30pt{ A \ar[r]^(0.4){\Delta}
      \ar[d]_{\Delta} & A \todot A
      \ar[d]^{\Delta \todot \Id} \\
      A\todot A \ar[r]_(0.4){\Id \todot \Delta} & A \todot A
      \todot A, } \qquad \xymatrix@R=15pt@C=25pt{A \todot A
      \ar[d]_{\epsilon \todot \Id} & A \ar[l]_(0.4){\Delta}
      \ar[r]^(0.4){\Delta} \ar[d]^{\Id} & A
      \todot A \ar[d]^{\Id \todot \epsilon} \\
      (B\rtimes \Gamma) \todot A \ar[r]^(0.6){\cong} & A
      &\ar[l]_(0.6){\cong} A \todot (B\rtimes \Gamma),}
\end{gather*}
\begin{gather*}
    \xymatrix@R=15pt@C=20pt{A \todot A \ar[d]_{S \todot \Id}
      && A \ar[ll]_(0.4){\Delta} \ar[rr]^(0.4){\Delta}
      \ar[d]^{\epsilon} && A
      \todot A \ar[d]^{\Id \todot S} \\
      A^{\co,\op} \todot A \ar[r]^(0.6){\check m} & A &
      B\rtimes \Gamma \ar[l]_(0.6){\check s}
      \ar[r]^(0.6){\hat r} & A &\ar[l]_(0.6){\hat m} A
      \todot A^{\co,\op}}
  \end{gather*}
 where the linear maps $\hat m, \check m, \hat r,
  \check s$ are given by
  \begin{align*}
    \hat m(a \todot a') &= aa' = \check m (a\todot a'), & \hat
    r(b\gamma) &= r(b), & \check s(\gamma b) &= s(b)
  \end{align*}
  for all $a,a'\in A, b\in B,\gamma\in \Gamma$.

  A morphism of $(B,\Gamma)$-Hopf algebroids
  $(A,\Delta_{A},\epsilon_{A},S_{A})$,
  $(C,\Delta_{C},\epsilon_{C},S_{C})$ is a  morphism of
  $(B,\Gamma)$-algebras $\pi \colon A \to C$ such that
  $\Delta_{C} \circ \pi  = (\pi \todot \pi) \circ
  \Delta_{A}$, $\epsilon_{C} \circ \pi = \epsilon_{A}$,
  $S_{C} \circ \pi = \pi^{\co,\op}\circ S_{A}$.
  We denote the category of all $(B,\Gamma)$-Hopf algebroids
  by $\Hopf_{(B,\Gamma)}$.
\end{definition}
  A $(B,\Gamma)$-Hopf algebroid reduces to an
  $\mathfrak{h}$-Hopf algebroid in the sense of
  \cite{koelink:su2} when $\mathfrak{h}$ is a commutative
  Lie algebra, $B$ is the algebra of meromorphic functions
  on the dual $\mathfrak{h}^{*}$, and
  $\Gamma=\mathfrak{h}^{*}$ acts by shifting the
  argument. Let us note that the axioms above can be
  weakened, see \cite[Proposition 2.2]{koelink:su2}, but our
  examples shall automatically satisfy the apparently
  stronger conditions above.
\begin{example} 
  The $(B,\Gamma)$-algebra $B \rtimes \Gamma$ is a
  $(B,\Gamma)$-Hopf algebroid, where $\Delta(b \gamma) =
  b\gamma \todot \gamma = \gamma \todot b\gamma$,
  $\epsilon(b\gamma) = b\gamma$, and $S( b\gamma) =
  \gamma^{-1}b$ for all $b\in B,\gamma\in \Gamma$.
\end{example}
Let us comment on some straightforward properties of
$(B,\Gamma)$-Hopf algebroids:
\begin{remarks}
  \label{remarks:bg-hopf}
    Let $(A,\Delta,\epsilon,S)$ be a $(B,\Gamma)$-Hopf
    algebroid.
    \begin{enumerate}
    \item If $\gamma\neq \gamma'$, then
      $\epsilon(A_{\gamma,\gamma'}) =0$ because $(B\rtimes
      \Gamma)_{\gamma,\gamma'} =0$.
     \item We have $\Delta(A)(1 \todot A_{e,*}) = A\todot A =
      (A_{*,e} \todot 1)\Delta(A)$, where
      $A_{e,*}=\sum_{\gamma} A_{e,\gamma}$ and $A_{*,e} =
      \sum_{\gamma} A_{\gamma,e}$. Indeed, by
      \cite[Proposition 1.3.7]{timmermann:measured},
      \begin{align*}
        \sum (xS(y_{(1)}) \todot 1)\Delta(y_{(2)}) &= x \todot y
        = \sum \Delta(x_{(1)})(1 \todot S(x_{(2)})y)
      \end{align*}
      for all $x\in A_{\gamma,\gamma'}, y\in
      A_{\gamma',\gamma''}$, $\gamma,\gamma',\gamma'' \in
      \Gamma$, where $\sum x_{(1)} \todot x_{(2)}=\Delta(x)$ and
      $\sum y_{(1)} \todot y_{(2)} = \Delta(y)$.
    \end{enumerate}
\end{remarks}
$(B,\Gamma)$-Hopf algebroids fit into the general definition
of Hopf algebroids \cite{boehm:algebroids} as
follows.
\begin{remark} Let
  $(A,\Delta,\epsilon,S)$ be a $(B,\Gamma)$-Hopf algebroid.
 Denote by ${_{\bullet}\epsilon}$ and
      $\epsilon_{\bullet}$ the compositions of
      $\epsilon\colon A\to B\rtimes \Gamma$ with the linear
      maps $B\rtimes \Gamma \to B$ given by $b\gamma \mapsto
      b$ and $\gamma b \mapsto b$, respectively, and denote
      by $_{\bullet}\Delta$ and $\Delta_{\bullet}$ the
      compositions of $\Delta$ with the natural inclusions
      $A \todot A \to {_{\bullet}} A \otimesB {_{\bullet}
        A}$ and $A \todot A \to {A_{\bullet}} \otimesB
      {A_{\bullet}}$, respectively (see Remark
      \ref{remark:takeuchi}).
    \begin{enumerate}
    \item The maps $_{\bullet }\epsilon,
      \epsilon_{\bullet}\colon A\to B$ will in general not
      be homomorphisms, but satisfy
      \begin{align*}
        {_{\bullet}\epsilon}(ar({_{\bullet}\epsilon}(a'))&=
        {_{\bullet}\epsilon}(aa') =
        {_{\bullet}\epsilon}(as({_{\bullet}\epsilon}(a')), &
        \epsilon_{\bullet}(r(\epsilon_{\bullet}(a))a') &=
        \epsilon_{\bullet}(aa') =
        \epsilon_{\bullet}(s(\epsilon_{\bullet}(a))a')
      \end{align*}
      for all $a,a' \in A$. Indeed, since $\epsilon(a) =
      {_{\bullet}\epsilon}(a)\partial_{a}$ for all
      homogeneous $a'\in A$,
      \begin{align*}
        {_{\bullet}\epsilon}(aa') \partial_{aa'} &=
        {_{\bullet}\epsilon}(a) \partial_{a} \cdot
        {_{\bullet}\epsilon}(a') \partial_{a'}
        =\partial_{a}( {_{\bullet}\epsilon}(a'))
        {_{\bullet}\epsilon}(a) \partial_{a}\partial_{a'} =
        {_{\bullet}\epsilon}(ar({_{\bullet}\epsilon}(a'))) \partial_{aa'}
      \end{align*}
      for all homogeneous $a,a'\in A$, and the remaining
      equations follow similarly.
    \item One easily verifies that
      $({_{\bullet}A},{_{\bullet
        }\Delta},{_{\bullet}\epsilon})$,
      $(A_{\bullet},\Delta_{\bullet},\epsilon_{\bullet})$
      are $B$-corings, ${_{\bullet}\mathcal{A}}
      :=(A,{_{\bullet}\Delta},{_{\bullet}\epsilon})$ is a
      left $B$-bialgebroid, and
      $\mathcal{A}_{\bullet}:=(A^{\co},\Delta_{\bullet},\epsilon_{\bullet})$
      is a right $B$-bialgebroid in the sense of
      \cite{boehm:algebroids}.  Using the relations $\check
      s \circ \epsilon = s \circ \epsilon_{\bullet}$ and
      $\hat r \circ \epsilon = r \circ
      {_{\bullet}\epsilon}$, one furthermore finds that
      $(\mathcal{A}_{\bullet},{_{\bullet}\mathcal{A}},S)$ is
      a Hopf algebroid over $B$. To make the match with
      Definition 4.1 in \cite{boehm:algebroids}, one has to
      take
      $H,s_{L},t_{L},\Delta_{L},\epsilon_{L},s_{R},t_{R},\Delta_{R},\epsilon_{R},S$
      equal to
      $A,s,r,{_{\bullet}\Delta},{_{\bullet}\epsilon},r,s,\Delta_{\bullet},\epsilon_{\bullet},S$,
      respectively.
  \end{enumerate}
\end{remark}



Let $B$ and $C$ be commutative algebras with a left action
of $\Gamma$ and let $\phi\colon B\to C$ be a
$\Gamma$-equivariant homomorphism. We then obtain base
change functors $\phi_{*}\colon \bgalg \to \cgalg$ and
$\phi_{*}\colon \Hopf_{(B,\Gamma)}
\to \Hopf_{(C,\Gamma)}$ as follows.
Let $A$ be a $(B,\Gamma)$-algebra.  Regard $C$ as  a
$B$-module via $\phi$, and $A$ as a
$B$-bimodule, where $b \cdot a \cdot b' = r(b)as(b')$
for all $b,b'\in B$, $a\in A$. Then the vector space
$\phi_{*}(A):=C \btimes A \btimes C$ carries the
structure of a $(C,\Gamma)$-algebra such that 
\begin{gather*}
  (c \btimes a \btimes d)
  (c' \btimes a' \btimes d') =
  c \partial^{r}_{a}(c') \btimes  aa' \btimes(\partial^{s}_{a'})^{-1}(d)d', \\
  \begin{aligned}
    \partial_{c\btimes a \btimes d} &= \partial_{a}, &
    (r\times s)(c\otimes c') &= c\btimes 1 \btimes c'
    &&\text{ for all } c,c',d,d' \in C, a,a' \in A.
  \end{aligned}
\end{gather*}
Every morphism of $(B,\Gamma)$-algebras $\pi\colon A \to A'$
evidently yields a morphism of $(C,\Gamma)$-algebras
$\phi_{*}(\pi) \colon \phi_{*}(A) \to \phi_{*}(A')$,
$c\btimes a \btimes c'\mapsto c \btimes \pi(a) \btimes c'$,
and the assignments $A \mapsto \phi_{*}(A)$ and $\phi
\mapsto \phi_{*}(\pi)$ form a functor $\phi_{*}\colon
\bgalg \to \cgalg$.
\begin{lemma} \label{lemma:bg-cb}
  \begin{enumerate}
  \item There exists a morphism of $(C,\Gamma)$-algebras
    $\phi^{(0)} \colon \phi_{*}(B\rtimes \Gamma) \to C\rtimes
    \Gamma$, $c \btimes b\gamma \btimes c'
    \mapsto c\phi(b) \gamma c' = c\phi(b)\gamma(c') \gamma$.
  \item For all $(B,\Gamma)$-algebras $A,D$, there exists a
    unique morphism $\phi^{(2)}_{A,D} \colon \phi_{*}(A
    \todot D) \to \phi_{*}(A) \todot \phi_{*}(D)$,
    $c\btimes (a \todot d) \btimes c' \mapsto (c\btimes
    a\btimes 1) \todot (1 \btimes d \btimes c')$.
  \end{enumerate}
\end{lemma}
\begin{proposition} \label{proposition:bg-cb}
  Let $(A,\Delta,\epsilon,S)$ be a $(B,\Gamma)$-Hopf
  algebroid. Then $\phi_{*}(A)$ is $(C,\Gamma)$-Hopf algebroid
  with respect to the morphisms
  \begin{enumerate}
  \item $\Delta'\colon \phi_{*}(A)
    \xrightarrow{\phi_{*}(\Delta)} \phi_{*}(A\todot A)
    \xrightarrow{\phi^{(2)}_{A,A}} \phi_{*}(A) \todot
    \phi_{*}(A)$, given by $c \btimes a \btimes c' \mapsto
    \sum_{i} (c\btimes a'_{i} \btimes 1) \todot (1
    \btimes a''_{i} \btimes c')$ whenever $\Delta(a)=\sum
    a'_{i}\todot a''_{i}$;
  \item $\epsilon'\colon \phi_{*}(A)
    \xrightarrow{\phi_{*}(\epsilon)} \phi_{*}(B\rtimes
    \Gamma) \xrightarrow{\phi^{(0)}} C\rtimes \Gamma$, given
    by $c \btimes a \btimes c' \mapsto \sum_{i} c
    \phi(b_{i})\gamma_{i} c'$ whenever $\epsilon(a)=\sum_{i}
    b_{i}\gamma_{i}$;
  \item $S' \colon \phi_{*}(A) \to
    (\phi_{*}A)^{\co,\op}$ given by $c\btimes a\btimes c'
    \mapsto c' \btimes S(a) \btimes c$.
  \end{enumerate}
\end{proposition}
The assignments $(A,\Delta,\epsilon,S) \mapsto
(\phi_{*}(A),\Delta',\epsilon',S')$ as above and $\pi \mapsto
\phi_{*}(\pi)$ evidently form a functor $\phi_{*} \colon
\Hopf_{(B,\Gamma)} \to \Hopf_{(C,\Gamma)}$.

The preceding definitions and results extend to $*$-algebras as follows.
Assume that $B$ is a $*$-algebra and that $\Gamma$ preserves
its involution.

A $(B,\Gamma)$-$*$-algebra is a $(B,\Gamma)$-algebra with an
involution that is compatible with the grading and the
involution on $B$, and a morphism of
$(B,\Gamma)$-$*$-algebras is a morphism of
$(B,\Gamma)$-algebra that preserves the involution.  We
denote by $\bgsalg$ the category of all
$(B,\Gamma)$-$*$-algebras. This subcategory of $\bgalg$ is
monoidal because the crossed product $B\rtimes \Gamma$ is a
$(B,\Gamma)$-$*$-algebra with respect to the involution
given by $(b\gamma)^{*}=\gamma^{-1} b^{*}$, and for all
$(B,\Gamma)$-$*$-algebras $A,C$, the fiber product $A\todot
C$ is a $(B,\Gamma)$-$*$-algebra with respect to the
involution given by $(a \todot c)^{*} =a^{*} \todot c^{*}$.
\begin{definition} \label{definition:bg-s-algebra} A
  $(B,\Gamma)$-Hopf $*$-algebroid is a $(B,\Gamma)$-Hopf
  algebroid $(A,\Delta,\epsilon,S)$ where $A$ is a
  $(B,\Gamma)$-$*$-algebra and $\Delta$ and $\epsilon$ are
  morphisms of $(B,\Gamma)$-$*$-algebras. A morphism of
  $(B,\Gamma)$-Hopf $*$-algebroids is a morphism of the
  underlying $(B,\Gamma)$-Hopf algebroid and
  $(B,\Gamma)$-$*$-algebras.  We denote by
  $\Hopf_{(B,\Gamma)}^{*}$ the category of all
  $(B,\Gamma)$-Hopf $*$-algebroids.
\end{definition}
\begin{remark}
  If $(A,\Delta,\epsilon,S)$ is a $(B,\Gamma)$-Hopf
  $*$-algebroid, then $* \circ S \circ * \circ S = \Id$; see
  \cite[Lemma 2.9]{koelink:su2}.
\end{remark}

We denote by $\overline{A}$ the conjugate algebra of a
complex algebra $A$; this is the set $A$ with conjugated
scalar multiplication and the same addition and
multiplication. Thus, the involution of a $*$-algebra $A$ is
an automorphism $A\to \overline{A}^{\op}$.
\begin{lemma} \label{lemma:bg-bar}
  The category $\bgsalg$ has an automorphism
  $\overline{(-)}$ such that for every
  $(B,\Gamma)$-$*$-algebra $A$ and every morphism
of   $(B,\Gamma)$-$*$-algebras   $\phi \colon A\to C$,
    \begin{align*}
      (\overline{A})_{\gamma,\gamma'} &=
      \overline{A_{\gamma,\gamma'}} \text{ for all }
      \gamma,\gamma'\in \Gamma, & r_{\bar A} &=
      r_{A}\circ \ast, & s_{\bar A} &= s_{A} \circ \ast, &
      \overline{\phi} &= \phi.
    \end{align*}
    Furthermore, $\overline{(-)} \circ \overline{(-)} =
    \Id$, $\overline{(-)} \circ (-)^{\op} = (-)^{\op} \circ
    \overline{(-)}$, $\overline{(-)} \circ (-)^{\co} =
    (-)^{\co} \circ \overline{(-)}$.
  \end{lemma}
\begin{remark}
  There exists an isomorphism $B\rtimes \Gamma \to
  \overline{B\rtimes \Gamma}$, $b\gamma \mapsto
  b^{*}\gamma$, and for each pair of
  $(B,\Gamma)$-$*$-algebras $A,C$, there exists
  an isomorphism $\overline{A \todot C} \to \overline{A}
  \todot \overline{C}$, $a\todot c \mapsto a\todot
  c$.
\end{remark}

Let also $C$ be a commutative $*$-algebra with a left action
of $\Gamma$ and let $\phi\colon B\to C$ be a
$\Gamma$-equivariant $*$-homomorphism. Then for every
$(B,\Gamma)$-$*$-algebra $A$, the $(C,\Gamma)$-algebra
$\phi_{*}(A)$ is a $(C,\Gamma)$-$*$-algebra with respect
to the involution given by $(c \btimes a \btimes c')^{*}
=(\partial^{r}_{a})^{-1}(c)^{*} \btimes a^{*} \btimes
 \partial^{s}_{a}(c')^{*}$, and we obtain a functor
$\phi_{*} \colon \bgsalg \to\cgsalg$.  Likewise, we obtain
a functor $\phi_{*} \colon \Hopf_{(B,\Gamma)}^{*}\to
\Hopf_{(C,\Gamma)}^{*}$.

\subsection{The case of a trivial base algebra}
\label{section:trivial}

Assume for this subsection that $B=\complex$ equipped with the
trivial action of $\Gamma$.  Then the category of all
$(\complex,\Gamma)$-Hopf algebroids is equivalent to the
comma category of all Hopf algebras over $\complex
\Gamma$ as follows.

Recall that the group algebra $\complex \Gamma$ is a Hopf
$*$-algebra with involution, comultiplication, counit and
antipode given by $\gamma^{*}=\gamma^{-1}$,
$\Delta_{\complex \Gamma}(\gamma)=\gamma\otimes \gamma$,
$\epsilon_{\complex \Gamma}(\gamma)=1$, $S_{\complex
  \Gamma}(\gamma) = \gamma^{-1}$ for all $\gamma\in \Gamma
\subset \complex \Gamma$. Objects of the comma category
$\Hopf_{\complex \Gamma}$ are pairs consisting of a Hopf
algebra $A$ and a morphism of Hopf algebras $A \to \complex
\Gamma$, and morphisms from $(A,\pi_{A})$ to $(C,\pi_{C})$
are all morphisms $A \xrightarrow{\phi} C$ such that
$\pi_{C} \circ \phi = \pi_{A}$. Likewise, we define the
comma category $\Hopf^{*}_{\complex \Gamma}$ of Hopf
$*$-algebras over $\complex \Gamma$.

Note that a $(\complex,\Gamma)$-algebra is just a
$\Gamma\times\Gamma$-graded algebra and $A\todot C = A
\stackrel{\Gamma}{\otimes} C \subseteq A\otimes C$ for all
$(\complex,\Gamma)$-algebras $A,C$. Moreover, $\complex
\rtimes \Gamma = \complex \Gamma$, and for every
$(\complex,\Gamma)$-algebra $A$, the isomorphisms $(\complex
\rtimes \Gamma) \todot A \to A$ and $A \todot (\complex
\rtimes \Gamma) \to A$ are equal to $\epsilon_{\complex
  \Gamma} \otimes \Id$ and $\Id \otimes
\epsilon_{\complex \Gamma}$.
\begin{lemma} \label{lemma:b-to-hopf}
  Let $(A,\Delta,\epsilon,S)$ be a $(\complex,\Gamma)$-Hopf
  algebroid and let $\epsilon':=\epsilon_{\complex \Gamma}
  \circ \epsilon \colon A \to \complex$. The
  $(A,\Delta,\epsilon',S)$ is a Hopf algebra and $\epsilon
  \colon A \to \complex \Gamma$ is morphism of Hopf algebras.
\end{lemma}
\begin{proof}
  The preceding observations easily imply that
  $(A,\Delta,\epsilon',S)$ is a Hopf algebra.  To see that
  $\epsilon$ is a morphism of Hopf algebras, use the
  fact that $\Delta,\epsilon,S$ are $\Gamma\times
  \Gamma$-graded.
\end{proof}
\begin{lemma}\label{lemma:b-from-hopf}
  Let $(A,\Delta,\epsilon,S)$ be a Hopf algebra with a
  morphism $\pi \colon A \to \complex \Gamma$. Then $A$ is a
  $(\complex,\Gamma)$-algebra with respect to the grading
  given by $A_{\gamma,\gamma'} = \{ a\in A : (\pi \otimes
  \Id \otimes \pi)(\Delta(a)) = \gamma \otimes a \otimes
  \gamma'\}$ for all $\gamma,\gamma' \in \Gamma$, and
  $(A,\Delta,\pi,S)$ is a $(\complex,\Gamma)$-Hopf
  algebroid.
\end{lemma}
\begin{proof}
  The formula above evidently defines a
  $\Gamma\times\Gamma$-grading on $A$. Coassociativity of
  $\Delta$ implies that $\Delta(A) \subseteq A \todot
  A$. The remark preceding Lemma \ref{lemma:b-to-hopf} and
  the relation $\epsilon_{\complex \Gamma} \circ \pi =
  \epsilon$ imply $(\pi \todot \Id)\circ \Delta =\Id =
  (\Id \todot \pi) \circ \Delta$. Finally, in the notation
  of Definition \ref{definition:bg-hopf}, $\check m \circ (S
  \todot \Id) \circ \Delta =m \circ (S\otimes \Id) \circ
  \Delta = \epsilon = \check s \circ \pi$ and similarly
  $\hat m \circ (\Id \todot S) \circ \Delta = \hat r \circ
  \pi$.
\end{proof}
Putting everything together, one easily verifies:
\begin{proposition}
  There exists an equivalence of categories
  $\Hopf_{(\complex,\Gamma)}
  \stackrel{\bfF}{\underset{\bfG}{\rightleftarrows}}
  \Hopf_{\complex \Gamma}$, where
  $\bfF(A,\Delta,\epsilon,S)= ((A,\Delta,\epsilon_{\complex
    \Gamma}\circ \epsilon,S),\epsilon)$, $\bfF\phi= \phi$
  and $\bfG((A,\Delta,\epsilon,S),\pi)= (A,\Delta,\pi,S)$
  with the grading on $A$ defined as in Lemma
  \ref{lemma:b-from-hopf}, and $\bfG \phi= \phi$.
  Likewise, there exists an equivalence
  $\Hopf_{(\complex,\Gamma)}^{*} \rightleftarrows
  \Hopf^{*}_{\complex \Gamma}$.
\end{proposition}
Let us next consider the base change from $\complex$ to a
commutative algebra $C$ along the unital inclusion $\phi
\colon \complex \to C$ for a $(\complex,\Gamma)$-Hopf
algebroid $(A,\Delta,\epsilon,S)$.
\begin{remark}
 The action of $\Gamma$ on $C$ and the
  morphism $\epsilon \colon A\to\complex \Gamma$ turn $C$
  into a left module algebra over the Hopf algebra
  $(A,\Delta,\epsilon_{\complex \Gamma} \circ
  \epsilon,S)$, and $\phi_{*}(A,\Delta,\epsilon,S)$
  coincides with the Hopf algebroids considered in \cite[\S
  3.4.6]{boehm:algebroids} and \cite[Theorem
  3.1]{kadison:pseudo-hopf}, and is closely related to the
  quantum transformation groupoid considered in
  \cite[Example 2.6]{vainer}.
\end{remark}
Assume that $C$ is an algebra of functions on $\Gamma$ on which
  $\Gamma$ acts by left translations.
  \begin{proposition} \label{proposition:bg-lie} Define
    $m\colon A \to \End(A)$ and $m_{r},m_{s} \colon C
    \to \End(A)$ by $m(a')a=a'a$, $m_{r}(c)a = c
    (\partial^{r}_{a})a$, $m_{s}(c)a = c(\partial^{s}_{a})a$
    for all $a,a'\in A$, $c\in C$.  Then there exists a
    homomorphism $\lambda \colon \phi_{*}(A)
    \to \End(A)$, $c \otimes a \otimes c' \mapsto
    m_{r}(c)m(a)m_{s}(c')$, and $\lambda$ is injective if
    $aA_{\gamma,\gamma'} \neq 0$ for all non-zero $a \in A$
    and all $\gamma,\gamma' \in \Gamma$.
\end{proposition}
\begin{proof}
  First, note that
  \begin{align*}
    m(a')m_{r}(c)a = a' c( \partial^{r}_{a})a =
    c((\partial^{r}_{a'})^{-1}\partial^{r}_{a'a})a'a =
    m_{r}(\partial^{r}_{a'}(c)) m(a)a
  \end{align*}
 and likewise
  $m(a')m_{s}(c)=m_{s}(\partial^{s}_{a'}(c))m(a')$ for all
  $a,a'\in A$, $c\in C$. The existence of $\lambda$ follows.
 Assume that $aA_{\gamma,\gamma'} \neq 0$ for all
  non-zero $a \in A$ and all $\gamma,\gamma' \in
  \Gamma$. Let $d:=\sum_{i} c_{i} \otimes a_{i} \otimes
  c'_{i} \in \phi_{*}(A)$ be non-zero, where all $a_{i}$
  are homogeneous. Identifying $C\otimes A \otimes C$ with a
  space of $A$-valued functions on $\Gamma\times \Gamma$ and
  using the assumption, we first
  find $\gamma,\gamma' \in \Gamma$ such that $a:=\sum_{i}
  c_{i}(\partial^{r}_{a_{i}}\gamma)a_{i}c'_{i}(\gamma')$ is
  non-zero, and then an $a' \in
  A_{\gamma,\gamma'}$ such that $\lambda(d)a' = aa' \neq 0$.
\end{proof}
\begin{remark}\label{remark:bg-lie}
 Regard elements of $C$ as functionals on
  $\complex \Gamma$ via $c(\sum_{i} b_{i}\gamma_{i})=
  \sum_{i} b_{i}c(\gamma_{i})$.  Then $m_{r}(c)a = (c \circ
  \epsilon \otimes \Id)(\Delta(a))$, $m_{s}(c)a =(\Id
  \otimes c\circ \epsilon)(\Delta(a))$ for all $c\in C$,
  $a\in A$.
\end{remark}
\begin{example} \label{example:bg-lie} Let $G$ be a compact
  Lie group, $\mathcal{O}(G)$ its Hopf algebra of
  representative functions \cite[\S 1.2]{timmermann:buch} and $T
  \subseteq G$ a torus of rank $d$.  We now apply
  Proposition \ref{proposition:bg-lie}, where
  \begin{itemize}
  \item $A=\mathcal{O}(G)$, regarded as a Hopf
    $(\complex,\hat T)$-algebroid as in Lemma
    \ref{lemma:b-from-hopf} using the homomorphism $\pi
    \colon \mathcal{O}(G) \to \mathcal{O}(T)$ induced from
    the inclusion $T\subseteq G$, and the isomorphism
    $\mathcal{O}(T) \cong \complex \hat T$,
\item $C=U\mathfrak{t}$ is the enveloping algebra of the
 Lie algebra $\mathfrak{t}$ of $T$, regarded as a
  polynomial algebra of functions on $\hat T$ such that
  $X(\chi) = \frac{d}{dt}\big|_{t=0} \chi(e(tX))$, where
  $e\colon \mathfrak{t} \to T$ denotes the exponential map.
  \end{itemize}
  If we regard $U\mathfrak{t}$ as functionals on the algebra $\complex
  \hat T \cong \mathcal{O}(T)$ as in Remark \ref{remark:bg-lie},
  then $X(f)=\frac{d}{dt}\big|_{t=0} f(e(tX))$ and hence
  $m_{r},m_{s} \colon U\mathfrak{t}
  \to \End(\mathcal{O}(G))$ are given by
  \begin{align*}
    (m_{r}(X)a)(x) &= \frac{d}{dt}\Big|_{t=0}
    a(e(tX)x), & (m_{s}(X)a)(x) &=
    \frac{d}{dt}\Big|_{t=0} a(xe(tX))
  \end{align*}
  for all $X \in \mathfrak{t}$, $a\in \mathcal{O}(G)$, $x\in
  G$. Thus $\lambda(\mathcal{O}(G))
  \subseteq \End(\mathcal{O}(G))$ is the algebra generated
  by multiplication operators for functions in
  $\mathcal{O}(G)$ and by left and right differentiation
  operators along $T \subseteq G$.

  If $G$ is connected, then $\mathcal{O}(G)$ has no
  zero-divisors and hence $\lambda$ is injective as soon as
  for all $\chi,\chi' \in \hat T$, there exists some
  non-zero $a \in \mathcal{O}(G)$ such that
  $a(xyz)=\chi(x)a(y)\chi'(z)$ for all $x,z \in T$ and $y\in
  G$.
\end{example}

\subsection{Intertwiners for
  $(B,\Gamma)$-algebras}
\label{section:rn}

In this subsection, we study relations of the form used to
define the free orthogonal and free unitary dynamical
quantum groups $\Ao(\nabla,F)$ and $\Au(\nabla,F)$, and
show that such relations admit a number of natural
transformations. Conceptually, these relations express that
certain matrices are intertwiners or morphisms of
corepresentations, and the transformations correspond to
certain functors of  corepresentation categories.  Although
elementary, these observations provide short and systematic
proofs for the main results in the following subsection.

Regard $\Mn(B)$  as a  subalgebra
of $\Mn(B\rtimes \Gamma)$, and let $A$ be a
$(B,\Gamma)$-algebra.
Given a linear map $\phi \colon A\to C$ between algebras, we
denote by $\phi_{n} \colon \Mn(A) \to \Mn(C)$ its entry-wise
extension to $n\times n$-matrices.

\begin{definition} \label{definition:rn-intertwiner} A
  matrix $u \in \Mn(A)$ is \emph{homogeneous} if there are
  $\gamma_{1},\ldots,\gamma_{n} \in A$ such that $u_{ij} \in
  A_{\gamma_{i},\gamma_{j}}$ for all $i,j$. In that case,
  let $\partial_{u,i}:=\gamma_{i}$ for all $i$ and
  $\partial_{u}:=\diag(\gamma_{1},\ldots,\gamma_{n}) \in
  \Mn(B\rtimes \Gamma)$.  An \emph{intertwiner} for
  homogeneous matrices $u,v\in \Mn(A)$ is an $F\in \GLn(B)$
  satisfying $\partial_{v}F\partial_{u}^{-1} \in \Mn(B)$ and
  $r_{n}(\partial_{v}F\partial_{u}^{-1}) u = vs_{n}(F)$. We
  write such an intertwiner as $u\xrightarrow{F} v$ and let
  $\hat F:=\partial_{v} F\partial_{u}^{-1}$ if $u,v$ are
  understood.
\end{definition}
If $u\xrightarrow{F} v$ and $v\xrightarrow{G} w$ are
intertwiners, then evidently so are $v\xrightarrow{F^{-1}} u$ and $u
\xrightarrow{GF} w$.
\begin{definition}
  We denote by $\Rn(A)$ the category of all homogeneous
  matrices in $\Mn(A)$ together with their intertwiners as
  morphisms, and by $\Rn^{\times}(A)$ and
  $\Rn^{\times\top}(A)$ the full subcategories formed by all
  homogeneous $v$ in $\GLn(A)$ or $\GLn(A)^{\top}$,
  respectively.
\end{definition}
Evidently, $\Rn(A)$ is a groupoid, and the assignment
$A\mapsto \Rn(A)$ extends to a functor from
$(B,\Gamma)$-algebras to groupoids.

We shall make frequent use of the following straightforward relations.
\begin{lemma} \label{lemma:rn-grading}
  Let $u,v\in \Mn(A)$ be homogeneous,  $F\in \Mn(B)$ and
 $\hat F=\partial_{v}F\partial_{u}^{-1}$. Then
 \begin{align*}
 \hat F \in \Mn(B) \ \Leftrightarrow \ (F_{ij}=0 \text{ 
   whenever } \partial_{v,i} \neq \partial_{u,j}).
\end{align*}
Assume that these condition holds. Then $   \hat F =(\partial_{v,i}(F_{ij}))_{i,j} =
   (\partial_{u,j}(F_{ij}))_{i,j}$ and
 \begin{align*}
\hat F^{\top} &=
\partial_{u}F^{\top}\partial_{v}^{-1}, &
(\partial_{v}F)^{-\top} &= F^{-\top} \partial_{u}^{-1}, &
(F\partial_{u}^{-1})^{-\top} &= \partial_{v} F^{-\top}.
 \end{align*}
If  $B$ is a $*$-algebra and $\Gamma$ preserves
the involution, then
$\overline{\partial_{v}F} = \overline{F} \partial_{u}^{-1}$
and $\overline{F\partial_{u}} = \partial_{v}^{-1} \overline{F}$.
\end{lemma}

Given $u,v \in \Mn(A)$ such that $\partial^{s}_{u_{ik}}
= \partial^{r}_{v_{kj}}$ for all $i,k,j$, let
 $u \mal v:=
(\sum_{k} u_{ik} \todot v_{kj})_{i,j} \in \Mn(A\todot A)$.
\begin{lemma} \label{lemma:rn-ed}
  There exist functors
  \begin{align*}
    \bm{\epsilon} \colon \Rn(A) &\to \Rn(B\rtimes\Gamma), &&
    &u &\mapsto \partial_{u}, & (u\xrightarrow{F} v)
    &\mapsto
    (\partial_{u} \xrightarrow{F} \partial_{v}), \\
    \bm{\Delta} \colon \Rn(A) &\to \Rn(A\todot A), && & u
    &\mapsto u\mal u, & (u\xrightarrow{F} v) &\mapsto
    (u\mal u \xrightarrow{F} v\mal v),  \\
    (-)^{\op}\colon \Rn(A) &\to \Rn(A^{\op}), & & & u
    &\mapsto u^{\op}:=u, &
    (u\xrightarrow{F} v) &\mapsto (u^{\op} \xrightarrow{\hat
      F} v^{\op}),
  \end{align*}
  and $\partial_{u\mal u} = \partial_{u}$, $\partial_{u^{\op}}
  =\partial^{-1}_{u}$ for all $u\in \Rn(A)$.
\end{lemma}
\begin{proof}
  For each $u\in \Rn(A)$, the matrices $\partial_{u}, u\mal
  u,u^{\op}$ evidently are homogeneous, and for every
  intertwiner $u\xrightarrow{F} v$, Lemma
  \ref{lemma:rn-grading}  implies
  \begin{align*}
    r_{n}(\hat F)u \mal u &= vs_{n}(F) \mal u = v \mal
    r_{n}(\hat F)u =
    v\mal vs_{n}(F), \\
    r_{n}(\partial_{v^{\op}}\hat F \partial_{u^{\op}}^{-1})
    u^{\op} &=  r_{n}(F)^{\op}u^{\op} = (r_{n}(\hat
    F)u)^{\op} = (vs_{n}(F))^{\op} = v^{\op}s_{n}(\hat
    F)^{\op}. 
\end{align*}
Functoriality of the assignments is evident.
\end{proof}
\begin{lemma} \label{lemma:rn-intertwiner} There exist 
  contravariant functors
\begin{align*}
  (-)^{\top,\co} \colon \Rn(A) &\to \Rn(A^{\co}), & u
  &\mapsto u^{\top,\co} := u^{\top}, & (u \xrightarrow{F} v)
  &\mapsto (v^{\top,\co} \xrightarrow{F^{\top}}
  u^{\top,\co}),  \\
  (-)^{-\co} \colon \Rn^{\times}(A) &\to \Rn(A^{\co}), & u
  &\mapsto u^{-\co}:= u^{-1}, & (u \xrightarrow{F} v) &\mapsto
  (v^{-\co} \xrightarrow{\hat F^{-1}} u^{-\co}),
  \end{align*}
  and $\partial_{u^{\top,\co}} =\partial_{u}$  and
  $\partial_{u^{-\co}} = \partial_{u}^{-1}$ for all $u$.
\end{lemma}
\begin{proof}
  If $u \in \Rn(A)$, then $u^{\top,\co}$ evidently is
  homogeneous as claimed.  Assume $u \in
  \Rn^{\times}(A)$. We claim that $u^{-\co}$ is homogeneous
  and $\partial_{u^{-\co}} = \partial_{u}^{-1}$.  For each
  $i,j$, let $w_{ij}$ be the homogeneous part of
  $(u^{-1})_{ij}$ of degree
  $(\partial_{u,j}^{-1},\partial_{u,i}^{-1})$. Then
  $\sum_{l} u_{il}w_{lj}$ is homogeneous of degree
  $(\partial_{u,i}\partial_{u,j}^{-1},e)$ and coincides with
  the homogeneous part of the sum $\sum_{l}
  u_{il}(u^{-1})_{lj}$ of the same degree for each
  $i,j$. Hence, $uw=uu^{-1}$ and the claim follows.

  Let $u\xrightarrow{F} v$ be an intertwiner. Using Lemma
  \ref{lemma:rn-grading}, one easily verifies that
  \begin{align*}
    s_{n}(\partial_{u^{\top,\co}}F^{\top}\partial_{v^{\top,\co}}^{-1})v^{\top}
    &= s_{n}(\hat F)^{\top}v^{\top} = (vs_{n}(F))^{\top}
    = (r_{n}(\hat F)u)^{\top} = u^{\top}r_{n}(F^{\top}), \\
    s_{n}(\partial_{u^{-\co}} \hat
    F^{-1} \partial_{v^{-\co}}^{-1})v^{-1} &=
    s_{n}(F^{-1})v^{-1} =u^{-1}r_{n}(\hat F^{-1}). 
  \end{align*}
  Finally, functoriality of the assignments is easily
  checked.
\end{proof}
Forming suitable compositions, we obtain further co- or contravariant functors
\begin{align*}
    (-)^{-\top} &= (-)^{\top,\co}\circ (-)^{-\co}\colon
    \Rn^{\times}(A) \to \Rn^{\times\top}(A), &
    &\begin{cases}
      u \mapsto u^{-\top}:=(u^{-1})^{\top}, \\
      (u\xrightarrow{F} v) \mapsto (u^{-\top}
      \xrightarrow{\hat F^{-\top}} v^{-\top}),
    \end{cases}\\
    (-)^{-\bot} &= (-)^{-\co} \circ (-)^{\top,\co} \colon
    \Rn^{\times\top}(A) \to \Rn^{\times}(A), &
    &\begin{cases}
      u \mapsto u^{-\bot}:=(u^{\top})^{-1}, \\
      (u\xrightarrow{F} v) \mapsto (u^{-\bot}
      \xrightarrow{\hat F^{-\top}} v^{-\bot})
    \end{cases} 
\end{align*}
and
\begin{align*}
    (-)^{-\co,\op} = (-)^{\op}\circ (-)^{-\co} \colon
    \Rn^{\times}(A) \to \Rn(A^{\co,\op}),  \
    \begin{cases}
      u \mapsto (u^{-\co})^{\op}, \\ (u \xrightarrow{\! F\!} v)
      \mapsto (v^{-\co,\op} \xrightarrow{\!\! F^{-1}\!\! }
      u^{-\co,\op}),
    \end{cases}
  \end{align*}
where $\partial_{u^{-\co,\op}} =\partial_{u}$ and $\partial_{u^{-\top}} = \partial_{u^{-\bot}} = \partial_{u}^{-1}$ for all $u$. 
\begin{lemma} \label{lemma:rn-commute} The following
  relations hold:
  \begin{align*}
\mathrm{i)} & \  \  (-)^{\op} \circ (-)^{-\top} = (-)^{-\bot} \circ
    (-)^{\op},   &
\mathrm{ii) } &\ \ (-)^{-\top} \circ (-)^{\op} = (-)^{\op}
    \circ
    (-)^{-\bot}, \\
\mathrm{iii) } &  \ \  (-)^{-\top} \circ \bm{\Delta} = \bm{\Delta} \circ
    (-)^{-\top}, &
\mathrm{iv) } & \ \ (-)^{-\top} \circ (-)^{-\co,\op} =
    (-)^{-\co,\op} \circ (-)^{-\top}.
  \end{align*} 
\end{lemma}
\begin{proof}
  i) We first check that the compositions agree on
  objects. Let us write $v^{\op}$ if we regard $v\in \Mn(A)$
  as an element of $\Mn(A^{\op})$. Then map $\Mn(A) \to
  \Mn(A^{\op})$ given by $v \mapsto (v^{\top})^{\op} =
  (v^{\op})^{\top}$ is an antihomomorphism and hence
  $(v^{-\top})^{\op} = 
  (v^{\top,\op})^{-1} = (v^{\op})^{-\bot}$ for all $v \in
  \GLn(A)$.  The compositions also agree on morphisms
  because for every intertwiner $u\xrightarrow{F} v$, we
  have
  $\partial_{v^{-\top}}(\partial_{v}F\partial_{u}^{-1})^{-\top}\partial_{u^{-\top}}^{-1}
  =
    \partial_{v}^{-\top} \partial_{v}F^{-\top}\partial_{u}^{-1}
    \partial_{u} = F^{-\top}$. 

  ii) This equation follows similarly like i).  

  iii)  Let $u \in
  \Rn^{\times}(A)$. Then $(u \mal u)^{-\top} =u^{-\top}
  \mal u^{-\top}$ because
  \begin{align*}
    \sum_{k} (u\mal u)_{ik} (u^{-\top} \mal u^{-\top})_{jk}
    = \sum_{k,l,m} u_{il} (u^{-1})_{mj} \todot
    u_{lk}(u^{-1})_{km} = \delta_{i,j} 1 \todot 1.
 \end{align*}
 and similarly $\sum_{k} (u^{-\top} \mal u^{-\top})_{ki}
 (u\mal u)_{kj} = \delta_{i,j} 1\todot 1$. For morphisms, we
 have nothing to check because $\partial_{u\mal u}
 = \partial_{u}$.

 iv) This equation follows from the relation $(-)^{-\top}
 \circ (-)^{\op} \circ (-)^{-\co} = (-)^{\op} \circ
 (-)^{-\bot} \circ (-)^{-\co} = (-)^{\op} \circ (-)^{-\co}
 \circ (-)^{\top,\co} \circ (-)^{-\co}$.
\end{proof}

Assume for a moment that  $(A,\Delta,\epsilon,S)$ is a Hopf
  $(B,\Gamma)$-algebroid.
\begin{definition} \label{definition:rn-corep}
  A \emph{matrix corepresentation} of
  $(A,\Delta,\epsilon,S)$ is a $v\in \Rn(A)$ for some $n\in
  \naturals$ satisfying $\Delta_{n}(v)=v\mal v$,
  $\epsilon_{n}(v)=\partial_{v}$, $S_{n}(v)=v^{-1}$. 
\end{definition}
\begin{lemma} \label{lemma:rn-corep}
  If $v\xrightarrow{F} w$ is a morphism in $\Rn(A)$ and
  $v$ is a matrix corepresentation, then so is $w$.
\end{lemma}
\begin{proof}
  Applying the morphisms $\Delta,\epsilon,S$ and the
  functors $\bm{\Delta},\bm{\epsilon},(-)^{-\co,\op}$ to $v
  \xrightarrow{F} w$ or its inverse, we get intertwiners
  $w\mal w \xrightarrow{F^{-1}} v\mal v = \Delta_{n}(v)
  \xrightarrow{F} \Delta_{n}(w)$, $\partial_{w}
  \xrightarrow{F^{-1}} \partial_{v} = \epsilon_{n}(v)
  \xrightarrow{F} \epsilon_{n}(w)$ and $w^{-\co,\op}
  \xrightarrow{F^{-1}} v^{-\co,\op} = S_{n}(v)
  \xrightarrow{F} S_{n}(w)$.
\end{proof}

Let us now discuss the involutive case. 

Given a $*$-algebra $C$ and a matrix $v \in \Mn(C)$, we write
$\overline{v}:=(v^{*}_{ij})_{i,j} = (v^{*})^{\top}$.

Assume that $B$ is a $*$-algebra, that $\Gamma$ preserves
the involution, and that $A$ is a $(B,\Gamma)$-algebra. Then
there exists an obvious functor $\Rn(A) \to \Rn(\bar A)$,
given by $u \mapsto u$ and $(u \xrightarrow{F} v) \mapsto (u
\xrightarrow{\overline{F}} v)$. Composition with $(-)^{\op}$
gives a functor
\begin{align*}
  (-)^{\barop} \colon \Rn(A) &\to \Rn(\overline{A}^{\op}), &
  u &\mapsto u^{\barop}:= u^{\op}, & (u \xrightarrow{F} v)
  &\mapsto (u^{\barop} \xrightarrow{\overline{\hat F}}
  v^{\barop}),
  \end{align*}
  and $\partial_{u^{\barop}} = \partial_{u}^{-1}$ for all
  $u$. For later use, we note the following relation.
  \begin{lemma} \label{lemma:rn-barop} Let $u^{-\top}
    \xrightarrow{F} v$ be an intertwiner in
    $\Rn^{\times}(A)\cap \Rn^{\times,\top}(A)$. Then
    $(v^{\barop})^{-\top} \xrightarrow{F^{*}} u^{\barop}$ is
    an intertwiner in $\Rn^{\times}(\overline{A}^{\op})\cap
    \Rn^{\times,\top}(\overline{A}^{\op})$.
\end{lemma}
\begin{proof}
  Subsequent applications of the functors $(-)^{\barop}$,
  $(-)^{-\top}$ yield intertwiners $(v^{-\top})^{\barop} =
  (v^{\barop})^{-\bot} \xrightarrow{\overline{\hat F}^{-1}}
  u^{\barop}$ and $(u^{\barop})^{-\top}
  \xrightarrow{\overline{F}^{-\top}=F^{-*}} v^{\barop}$.
\end{proof}
Finally, assume that $A$ is a $(B,\Gamma)$-$*$-algebra. Then there exists  a functor 
\begin{align*}
    (-)^{*,\co} \colon  \Rn(A) &\to \Rn(A^{\co}), &
    u &\mapsto u^{*,\co}:=u^{*}, & 
    (u \xrightarrow{F} v) &\mapsto (v^{*,\co} 
    \xrightarrow{\hat F^{*}} u^{*,\co}),  
\end{align*}
because $s_{n}(F^{*})v^{*} =u^{*}r_{n}(\hat F^{*})$ for
every intertwiner $u\xrightarrow{F} v$, and $\partial_{u^{*,\co}}
= \partial_{u}^{-1}$. Composing with $(-)^{\top,\co}$ for
$A^{\co}$ and with  $(-)^{-\top}$, respectively, we get functors
\begin{align} \label{eq:rn-overline}
  \overline{(-)}\colon \Rn(A) &\to \Rn(A), & u\mapsto
  \overline{u} &= (u_{ij}^{*})_{i,j}, &
  (u\xrightarrow{F} v) &\mapsto (\overline{u}
  \xrightarrow{\overline{\hat F}} \overline{v}), \\
 \label{eq:rn-star} \Rn(A) &\to \Rn(A), &
  u\mapsto \overline{u} &= \overline{u}^{-\top} = \overline{u^{-\bot}}, &
  (u\xrightarrow{F} v) &\mapsto (\overline{u}^{-\top}
  \xrightarrow{F^{-*}} \overline{v}^{-\top}).
\end{align}

\subsection{The free orthogonal and free unitary dynamical
  quantum groups}

\label{section:ao}

Using the preparations of the last subsection, we now show
that the algebras $\Ao(\nabla,F)$ and $\Au(\nabla,F)$ are
$(B,\Gamma)$-Hopf algebroids as claimed in the introduction.

Let $B$ be a commutative algebra with an action of a group
$\Gamma$ as before, and let $\gamma_{1},\ldots,\gamma_{n}
\in \Gamma$ and $\nabla=\diag(\gamma_{1},\ldots,\gamma_{n})
\in \Mn(B\rtimes \Gamma)$.

Let $F \in \GLn(B)$ be \emph{$\nabla$-odd} in the sense that
$\nabla F \nabla \in \Mn(B)$. The first definition and
theorem in the introduction can be reformulated as follows.
\begin{definition}
  The \emph{free orthogonal dynamical quantum group over
  $B$ with parameters $(\nabla,F)$} is the universal
  $(B,\Gamma)$-algebra $\Ao(\nabla,F)$ with a $v\in
  \Rn^{\times}(\Ao(\nabla,F))$ such that $\partial_{v} =
  \nabla$ and $v^{-\top} \xrightarrow{F} v$ is an
  intertwiner.
\end{definition}
\begin{theorem} \label{theorem:ao-hopf} The
  $(B,\Gamma)$-algebra $\Ao(\nabla,F)$ can be equipped with
a unique structure of a $(B,\Gamma)$-Hopf algebroid such that
  $v$ becomes a matrix corepresentation.
\end{theorem}
\begin{proof}
  The existence of morphisms $\Delta \colon A\to A\todot A$,
  $\epsilon\colon A\to B\rtimes \Gamma$, $S \colon A\to
  A^{\co,\op}$ satisfying $\Delta_{n}(v) =v\mal v$,
  $\epsilon_{n}(v) =\nabla$, $S_{n}(v) =v^{-1}$ follows from
  the universal property of $A$ and the relations
  \begin{gather*}
      \bm{\Delta}(v^{-\top}\xrightarrow{F} v) =
    ((v\mal v)^{-\top} \xrightarrow{F} v\mal v), \quad
    \bm{\epsilon}(v^{-\top} \xrightarrow{F} v) =
    (\nabla^{-\top} \xrightarrow{F} \nabla), \\
  (v^{-\top} \xrightarrow{F} v)^{-\co,\op} =
  ( v^{-\co,\op}  \xrightarrow{F^{-1}} (v^{-\co,\op})^{-\top});
\end{gather*}
see Lemma \ref{lemma:rn-intertwiner} and
\ref{lemma:rn-commute}.  Straightforward calculations show
that $(A,\Delta,\epsilon,S)$ is a $(B,\Gamma)$-Hopf
algebroid.
\end{proof}
\begin{remarks} \label{remarks:ao}
  \begin{enumerate}
  \item In the definition of $\Ao(\nabla,F)$,  we may
    evidently assume that $\Gamma$ is generated by the
    diagonal components $\gamma_{1},\ldots,\gamma_{n}$ of
    $\nabla$.
  \item Denote by $B_{0} \subseteq B$ the smallest
    $\Gamma$-invariant subalgebra containing the entries of
    $F$ and $F^{-1}$, and by $\iota\colon B_{0}\to B$ the
    inclusion. Then there exists an obvious isomorphism
    $\Ao(\nabla,F) \cong\iota_{*}\Aoo(\nabla,F)$.
  \item Let $H \in \GLn(B)$ be $\nabla$-even and $\hat H=\nabla
    H \nabla^{-1}$. Then there exists an isomorphism $
    \Ao(\nabla,HF\hat
    H^{\top}) \to \Ao(\nabla,F)$ of $(B,\Gamma)$-Hopf
    algebroids whose extension to
    matrices sends $v \in \Ao(\nabla,HF\hat
    H^{\top})$ to $w:=r_{n}(\hat
    H) vs_{n}(H)^{-1}\in \Ao(\nabla,F)$. Indeed, there
    exists such a morphism of $(B,\Gamma)$-algebras because  in
    $\Ao(\nabla,F)$, we have intertwiners $v\xrightarrow{H}
    w$, $v^{-\top} \xrightarrow{\hat H^{-\top}}  w^{-\top}$
    and  $v^{-\top}\xrightarrow{F} v$, whence $w^{-\top}
    \xrightarrow{H F\hat H^{\top}} w$, and this morphism
    is compatible with $\Delta,\epsilon,S$ because  $w$ is a
    matrix corepresentation by Lemma \ref{lemma:rn-corep}. A
    similar argument yields the inverse of this morphism.
  \end{enumerate}
\end{remarks}

Assume that $B$ carries an involution which is preserved by
$\Gamma$, and let $F\in \GLn(B)$ be self-adjoint and
\emph{$\nabla$-even} in the sense that $\nabla F \nabla^{-1}
\in \Mn(B)$. The second definition and
theorem in the introduction can be reformulated as follows.
\begin{definition} \label{definition:intro-au-hopf} The
  \emph{free unitary dynamical quantum group over $B$ with
    parameters $(\nabla,F)$} is the universal
  $(B,\Gamma)$-$*$-algebra $\Au(\nabla,F)$ with a unitary $u
  \in \Rn^{\times}(\Au(\nabla,F))$ such that $\partial_{v} =
  \nabla$ and $(v^{-\top})^{-\top} \xrightarrow{F} v$ is an
  intertwiner.
\end{definition}
\begin{theorem}\label{theorem:intro-au-hopf}
  The $*$-algebra $\Au(\nabla,F)$ can be equipped with a unique
  structure of a $(B,\Gamma)$-Hopf $*$-algebroid such that
  $v$ becomes a matrix corepresentation.
\end{theorem}
To prove this result, we introduce an auxiliary
$(B,\Gamma)$-algebra which does not involve the involution
on $B$.
\begin{definition} \label{definition:au-prime} We denote by
  $\Aup(\nabla,F)$ the universal $(B,\Gamma)$-algebra with
  $v,w \in \Rn^{\times}(A)$ such that $\partial_{v}=\nabla$,
  $\partial_{w}=\nabla^{-1}$ and $v^{-\top} \xrightarrow{1}
  w$, $w^{-\top} \xrightarrow{F} v$ are intertwiners.
\end{definition}
Using the same techniques as in the proof of Theorem
\ref{theorem:ao-hopf}, one finds:
\begin{proposition} \label{proposition:au-prime-hopf} The
  $(B,\Gamma)$-algebra $\Aup(\nabla,F)$ can be equipped with
  a unique structure of a $(B,\Gamma)$-Hopf algebroid such
  that $v$ and $w$ become matrix corepresentations.
\end{proposition}
\begin{proposition} \label{proposition:au-prime-star} The
  $(B,\Gamma)$-algebra $\Aup(\nabla,F)$ can be equipped with
  an involution such that it becomes a $(B,\Gamma)$-Hopf
  $*$-algebroid and $w=\bar v$.
\end{proposition}
\begin{proof}
  Let $A:=\Aup(\nabla,F)$.  By Lemma \ref{lemma:rn-barop},
  we have intertwiners $(w^{\barop})^{-\top} \xrightarrow{1}
  v^{\barop}$ and $(v^{\barop})^{-\top}
  \xrightarrow{F^{*}=F} w^{\barop}$. The universal property
  of $A$ yields a homomorphism $j\colon A\to
  \overline{A}^{\op}$ satisfying $j_{n}(v) = w^{\barop}$ and
  $j_{n}(w) = v^{\barop}$.  Composition of $j$ with the
  canonical map $\bar A^{\op}\to A$ yields the desired
  involution, which is easily seen to be compatible with the
  comultiplication and counit.
\end{proof}
 Theorem \ref{theorem:intro-au-hopf} now is an immediate corollary to the
 following result:
 \begin{theorem} \label{theorem:au-prime-iso}
   There exists a unique $*$-isomorphism $\Au(\nabla,F) \to
   \Aup(\nabla,F)$ whose extension to matrices sends $u$ to $v$.
 \end{theorem}
 \begin{proof}
   One easily verifies that the universal properties of
   $A:=\Au(\nabla,F)$ and $A':=\Aup(\nabla,F)$ yield
   homomorphisms $A \to A'$ and $A' \to A$ whose extensions
   to matrices satisfy $u \mapsto v$ and
   $v\mapsto u$, $w \mapsto \bar u$, respectively.
 \end{proof}

 The following analogues of Remarks \ref{remarks:ao} apply
 to $\Au(\nabla,F)$:
 \begin{remarks} \label{remarks:au}
   \begin{enumerate}
   \item We may assume
     that $\Gamma$ is generated by the diagonal components
     of $\nabla$, and if $\iota \colon B_{0}\hookrightarrow
     B$ denotes the inclusion of the smallest
     $\Gamma$-invariant $*$-subalgebra containing the
     entries of $F$ and $F^{-1}$, then $\Au(\nabla,F) \cong
     \iota_{*}\Auo(\nabla,F)$.
   \item Let $H \in \GLn(B)$ be $\nabla$-even and unitary,
and     let $\hat H=\nabla H \nabla^{-1}$. Then there exists an isomorphism $
     \Au(\nabla,HFH^{*}) \to \Au(\nabla,F)$ of $(B,\Gamma)$-Hopf
     algebroids whose extension to matrices sends $u \in
     \Au(\nabla,HFH^{*})$ to $z:=r_{n}(\hat H) us_{n}(H)^{-1}\in
     \Au(\nabla,F)$. Indeed, there exists such a morphism of
     $(B,\Gamma)$-algebras because $z$ is a product of
     unitaries and in
     $\Au(\nabla,F)$, we have intertwiners $u\xrightarrow{H}
     z$, $\bar u^{-\top} \xrightarrow{H^{-*}} \bar
     z^{-\top}$ by \eqref{eq:rn-star}, and $\bar
     u^{-\top}\xrightarrow{F} u$, whence $\bar z^{-\top}
     \xrightarrow{H FH^{*}} z$, and this morphism is
     compatible with $\Delta,\epsilon,S$ because $z$ is a
     matrix corepresentation by Lemma
     \ref{lemma:rn-corep}. A similar argument yields the
     inverse of this morphism.
   \end{enumerate}
 \end{remarks}


 We finally consider involutions on certain quotients of
 $\Ao(\nabla,F)$.

Assume that $F,G \in \GLn(B)$ are $\nabla$-odd and $GF^{*}
=FG^{*}$.  Let $Q := G(\nabla \bar G \nabla)$.
\begin{definition} \label{definition:matrix-hopf-involution}
The \emph{free orthogonal
    dynamical quantum group} over $B$ with parameters
    $(\nabla,F,G)$ is  the universal 
  $(B,\Gamma)$-algebra  $\Ao(\nabla,F,G)$ with a $v\in \Rn^{\times}(A)$ such that
  $\partial_{v} =\nabla$ and $v^{-\top} \xrightarrow{F} v$ and
  $v \xrightarrow{Q} v$ are intertwiners.
\end{definition}
The algebra $\Ao(\nabla,F,G)$ depends only on $Q$ and not on
$G$, but shall soon be equipped with an involution that does
depend on $G$.  

Evidently, there exists a canonical quotient map $\Ao(\nabla,F) \to
\Ao(\nabla,F,G)$, and
\begin{align*}
  \Ao(\nabla,F,G) &\cong \Ao(\nabla,F) / (r(q)-s(q)) \quad \text{if } Q=\diag(q,\ldots,q),\\
  \Ao(\nabla,F,G) &\cong \Ao(\nabla,F) / (r(q_{i})-s(q_{j})|
i,j=1,\ldots,n) \quad \text{if } Q=\diag(q_{1},\ldots,q_{n}),
\end{align*}
because in the second case $(r_{n}(\hat
Q)v)_{ij}=r_{n}(\gamma_{i}(q_{i}))v_{ij} =
v_{ij}r_{n}(q_{i})$ and $(vs_{n}(Q))_{ij} =
v_{ij}s_{n}(q_{j})$ in $\Ao(\nabla,F)$ for all $i,j$.
\begin{theorem} \label{theorem:ao-prime-hopf} The
  $(B,\Gamma)$-algebra $\Ao(\nabla,F,G)$ can be equipped
  with a unique structure of a $(B,\Gamma)$-Hopf $*$-algebroid
  such that $\bar v \xrightarrow{G} v$ becomes an intertwiner and
  $v$  a matrix corepresentation.
\end{theorem}
\begin{proof} 
  The existence of $\Delta,\epsilon,S$ follows similarly as
  in the case of $\Ao(\nabla,F)$; one only needs to observe
  that additionally, application of the functors
  $\bm{\Delta},\bm{\epsilon}$ and $(-)^{\co,\op}$ to the
  intertwiner $(v \xrightarrow{Q} v)$ yield intertwiners
  $(v\mal v \xrightarrow{Q} v\mal v)$, $(\nabla
  \xrightarrow{Q} \nabla)$ and $((v^{-1})^{\co,\op}
  \xrightarrow{Q} (v^{-1})^{\co,\op})$
  
  Let us prove existence of the involution.  Let $w:=
  r_{n}(\nabla G \nabla)^{-1} v s_{n}(G)$.  Then there exist
  intertwiners
  \begin{gather*}
    w \xrightarrow{G} v, \qquad (v \xrightarrow{G^{-1}}
    w)\circ (v \xrightarrow{Q} v)
    = (v \xrightarrow{\nabla \bar G\nabla} w),  \\
    (v \xrightarrow{\nabla \bar G\nabla} w)^{\barop}\circ (w
    \xrightarrow{G} v)^{\barop} = \ (v^{\barop}
    \xrightarrow{G} w^{\barop}) \circ (w^{\barop}
    \xrightarrow{\nabla \bar G \nabla} v^{\barop}) =
    (w^{\barop} \xrightarrow{Q} w^{\barop}), \\
      (v^{-\top} \xrightarrow{F} v) \circ (w
      \xrightarrow{(\nabla \bar G \nabla)^{-1}} v)^{-\top} =
      (v^{-\top} \xrightarrow{F} v)\circ (w^{-\top}
      \xrightarrow{G^{*}} v^{-\top}) = (w^{-\top}
      \xrightarrow{FG^{*}=GF^{*}} v), \\
      ((v\xrightarrow{G^{-1}} w) \circ (w^{-\top}
      \xrightarrow{GF^{*}} v))^{\barop} = (w^{-\top}
      \xrightarrow{F^{*}} w)^{\barop} =
      ((w^{\barop})^{-\top} \xrightarrow{F} w^{\barop});
  \end{gather*}
  where we used Lemma \ref{lemma:rn-intertwiner} in the last
  line. The universal property of $A:=\Ao(\nabla,F,G)$
  therefore yields a homomorphism $j\colon A\to
  \overline{A}^{\op}$ such that $j_{n}(v)=w^{\barop}$, and
  this $j$ corresponds to a conjugate-linear
  antihomomorphism $A \to A$, $a \mapsto a^{*}$.  To see
  that the map $a\mapsto a^{*}$ is involutive, we only need
  to check $\overline{w}=v$.  The functor
  $\overline{(-)}$ of \eqref{eq:rn-overline} applied to $w
  \xrightarrow{G} v$ yields $\overline{w}
  \xrightarrow{\overline{\hat G}=\nabla \bar G\nabla}
  \overline{v} = w$, and composition with
  $w\xrightarrow{G} v$ gives $\overline{w} \xrightarrow{Q}
  v$. Hence, $\overline{w}=v$.

  Finally, the involution is compatible with the
  comultiplication and counit because $w$ is a matrix
  corepresentation by Lemma \ref{lemma:rn-corep}.
\end{proof}
\begin{remarks} \label{remarks:ao-fg}
  \begin{enumerate}
  \item The canonical quotient map $\Ao(\nabla,F) \to
    \Ao(\nabla,F,G)$ is a morphism of $(B,\Gamma)$-Hopf
    algebroids.
  \item  Analogues of Remarks
    \ref{remarks:ao} and \ref{remarks:au} apply to
    $\Ao(\nabla,F,G)$.
  \item Note that $\Ao(\nabla,F,G)$ is the universal
    $(B,\Gamma)$-$*$-algebra with a $v\in \Rn^{\times}(A)$
    such that $\partial_{v} =\nabla$ and $v^{-\top}
    \xrightarrow{F} v$ and $\bar v \xrightarrow{G} v$ are
    intertwiners. Indeed, the composition of $\bar v
    \xrightarrow{G} v$ with its image under the functor
    $\overline{(-)}$ in \eqref{eq:rn-overline} yields
    $v\xrightarrow{Q} v$.
 \item If $F=G$, then $\overline{v} = v^{-\top}$ and hence $v$ is
    unitary.  In general, assume that $H \in \GLn(B)$
    satisfies $\nabla H \nabla^{-1} \in \Mn(B)$ and
    $\overline{H}H^{\top}\in \complex \cdot G^{-1}F$.  Then
    $u:=r_{n}(H^{-1})vs_{n}(\nabla^{-1}H\nabla)$ is a
    unitary matrix corepresentation whose entries generate
    $\Ao(\nabla,F,G)$ as a $(B,\Gamma)$-algebra.  Indeed, $u
    \xrightarrow{\nabla^{-1}H\nabla} v$ is an intertwiner,
    and applying $\overline{(-)}$ and $(-)^{-\top}$,
    respectively, we get $\overline{u}
    \xrightarrow{\overline{H}} \overline{v} \xrightarrow{G}
    v \xrightarrow{F^{-1}} v^{-\top} \xrightarrow{H^{\top}}
    u^{-\top}$ which is scalar by assumption so that
    $\overline{u}=u^{-\top}$.
  \end{enumerate}
\end{remarks}


We finally consider a simple example; a more complex one is
considered in \S\ref{section:sud}.
\begin{example} \label{example:ao-su} Equip $\complex[X]$
  with an involution such that $X^{*}=X$ and an action of
  $\integers$ such that $X \stackrel{k}{\mapsto} X-k$ for
  all $k \in \integers$, and let
  $\gamma_{1}=1$, $\gamma_{2}=-1$, 
  $\nabla = \diag(\gamma_{1},\gamma_{2})$ and $
    F=G= \begin{pmatrix} 0 & -1 \\ 1 & 0
    \end{pmatrix}$. Then
    $A^{\complex[X]}_{\mathrm{o}}(\nabla,F,G) \cong
    \iota_{*}(A^{\complex}_{\mathrm{o}}(\nabla,F,G))$, where
    $\iota \colon \complex \to \complex[X]$ is the canonical
    map.  

    The algebra $A^{\complex}_{\mathrm{o}}(\nabla,F,G)$
    equipped with $\Delta,\epsilon_{\complex \Gamma} \circ
    \epsilon,S$ is a Hopf $*$-algebra by Lemma
    \ref{lemma:b-to-hopf}.  It is generated by the entries
    of a unitary matrix $v$ which satisfies $\overline{v} =
    G^{-1}vG$ and therefore has the form $v=
 \begin{pmatrix}
   \alpha & -\gamma^{*} \\ \gamma & \alpha^{*}
 \end{pmatrix}$. The relations $vv^{*}=1=v^{*}v$ then imply
 that $\alpha,\alpha^{*},\gamma,\gamma^{*}$ commute and
 $\alpha\alpha^{*}+\gamma\gamma^{*}=1$. Therefore,
 $A^{\complex}_{\mathrm{o}}(\nabla,F,G)$ is isomorphic to
 the Hopf $*$-algebra $\mathcal{O}(\SU)$ of representative
 functions on $\SU$.  

 The algebra $A^{\complex[X]}_{\mathrm{o}}(\nabla,F,G)\cong
 \iota_{*}(A^{\complex}_{\mathrm{o}}(\nabla,F,G))$ can be
 identified with the subalgebra of $\End(\mathcal{O}(\SU))$
 generated by multiplication operators  associated to
 elements of $\mathcal{O}(\SU)$ and left or right invariant
 differentiation operators along the diagonal torus in
 $\SU$; see Example \ref{example:bg-lie}.
\end{example}

 \subsection{The square of the antipode and the  scaling character groups}

\label{section:s}

The square of the antipode on the free dynamical quantum
groups $\Ao(\nabla,F)$, $\Au(\nabla,F)$,
$\Ao(\nabla,F,G)$ can be described in terms of certain
character groups as follows.

 Recall the isomorphisms of Lemma \ref{lemma:bg-monoidal}
 iii) and the anti-automorphism $S^{B\rtimes \Gamma}$ of
 $B\rtimes \Gamma$ given by $b\gamma \mapsto \gamma^{-1} b$.
\begin{definition} \label{definition:s-group} Let
  $(A,\Delta$, $\epsilon,S)$ be a $(B,\Gamma)$-Hopf
  algebroid.  A \emph{character group} on $A$ is a family of
  morphisms $\theta=(\theta^{(k)} \colon A\to B\rtimes
  \Gamma)_{k\in \integers}$ satisfying $(\theta^{(k)} \todot
  \theta^{(l)})\circ \Delta = \theta^{k+l}$, $ \theta^{(0)}=
  \epsilon$ and $\theta^{(k)} \circ S = S^{B\rtimes \Gamma}
  \circ \theta^{(-k)}$ for all $k,l\in \integers$.  We call
  a character group $\theta$ \emph{scaling} if $S^{2} =
  (\theta^{(1)} \todot \Id \todot \theta^{(-1)}) \circ
  \Delta^{(2)}$, where $\Delta^{(2)} = (\Delta \todot \Id)
  \circ \Delta = (\Id \todot \Delta) \circ \Delta$.
\end{definition}
We construct scaling character groups using
intertwiners of the form $u \xrightarrow{H}S^{2}_{n}(u)$ for suitable
matrix corepresentations $u$.
\begin{lemma} \label{lemma:s-group} Let
  $(A,\Delta,\epsilon,S)$ be a $(B,\Gamma)$-Hopf algebroid,
  let $\theta=(\theta^{(k)} \colon A\to B\rtimes
  \Gamma)_{k\in \integers}$ be a family of morphisms
  satisfying $(\theta^{(k)} \todot \theta^{(l)})\circ \Delta
  = \theta^{k+l}$ for all $k,l \in \integers$, and let $u
  \in \Rn^{\times}(A)$ be a matrix corepresentation.
  \begin{enumerate}
  \item $S^{2}_{n}(u) = (u^{-\top})^{-\top}$.
  \item Let $H =\partial_{u}^{-1}
    \theta^{(1)}_{n}(u) $. Then $H \in
    \GLn(B)$, $\partial_{u}H\partial_{u}^{-1} \in \Mn(B)$
    and $\theta^{(k)}_{n}(u) = \partial_{u} H^{k}$ for all
    $k \in \integers$.
  \item $\theta_{n}^{(0)}=\epsilon_{n}(u)$ and
    $\theta_{n}^{(k)}(S_{n}(u)) = S_{n}^{B\rtimes
      \Gamma}(\theta^{(-k)}_{n}(u))$ for all $k \in
    \integers$.
  \item  $S_{n}^{2}(u) = ((\theta^{(1)} \todot \Id \todot
    \theta^{(-1)})\circ\Delta^{(2)})_{n}(u)$ if and only if
    $u\xrightarrow{H} S^{2}_{n}(u)$ is an intertwiner.
  \end{enumerate}
\end{lemma}
\begin{proof}
  i) The map $\Mn(A) \to \Mn(A)$ given by $x \mapsto
  S_{n}(x)^{\top}$ is an antihomomorphism and therefore
  preserves inverses. Hence, $ S_{n}^{2}(u) =
  S_{n}(u^{-\top})^{\top} = (S_{n}(u)^{\top})^{-\top} =
  (u^{-\top})^{-\top}$.

  ii)  Since each  $\theta^{(k)}$ preserves the grading,
  there exists a family $(H_{k})_{k\in \integers}$ of
  elements of $\GLn(B)$ satisfying $\partial_{u}H_{k}\partial_{u}^{-1}  \in \Mn(B)$ and $\theta^{(k)}_{n}(u) = \partial_{u}H_{k}$ for all $k \in \integers$. The assumption on
  $\theta$ implies that $H_{k}H_{l} = H_{k+l}$ for all
  $k,l\in \integers$, and consequently, $H_{k}=H_{1}^{k}$ for
  all $k\in  \integers$.

  iii) By ii), $\theta_{n}^{(0)}(u)=\partial_{u}
  =\epsilon_{n}(u)$ and
  \begin{align*}
  \theta_{n}^{(k)}(S_{n}(u)) =
  \theta_{n}^{(k)}(u^{-1}) = \theta_{n}^{(k)}(u)^{-1} =
  H^{-k} \partial_{u}^{-1} = S_{n}^{B\rtimes
    \Gamma}(\partial_{u} H^{-k}) = S_{n}^{B\rtimes
    \Gamma}(\theta_{n}^{(-k)}(u)).  
  \end{align*}

  iv)  This follows from the relation
  \begin{align*}
    ((\theta^{(1)} \todot \Id \todot
    \theta^{(-1)})\circ\Delta^{(2)})_{n}(u) &= 
    \theta^{(1)}_{n}(u) \mal u \mal \theta^{(-1)}_{n}(u) \\ &=
    \partial_{u}H \mal u \mal \partial_{u}H^{-1}  =
    r_{n}(\partial_{u}H \partial_{u}^{-1}) u
    s_{n}(H^{-1}). \qedhere
  \end{align*}
\end{proof}

We first apply the lemma above to $\Ao(\nabla,F)$.
\begin{proposition} \label{proposition:s-ao} Let $F \in
  \GLn(B)$ be $\nabla$-odd. Then $\Ao(\nabla,F)$ has an
  intertwiner $v\xrightarrow{H} S^{2}(v)$ and a scaling
  character group $\theta$ such that $\theta^{(k)}_{n}(v) =
  \nabla H^{k}$ for all $k\in \integers$, where $H=(\nabla F
  \nabla)^{\top}F^{-1}$. 
\end{proposition}
 \begin{proof}
   By Lemma \ref{lemma:s-group} i), $(v^{-\top}
   \xrightarrow{F}v) \circ (v^{-\top} \xrightarrow{F}
   v)^{-\top} = S^{2}(v)\xrightarrow{H^{-1}} v$. To construct
   $\theta$,  let $k \in \integers$ and $x=\nabla
   H^{k}$. By Lemma \ref{lemma:rn-grading}, $x^{-\top} = (H^{-\top})^{k}
   \nabla^{-1}$ and hence
  \begin{align*}
    (\nabla F \nabla) x^{-\top} = \nabla F \nabla
    (\nabla^{-1} F^{-1} \nabla^{-1}F^{\top})^{k} \nabla^{-1}
    = \nabla (\nabla^{-1}F^{\top}\nabla^{-1}F^{-1})^{k} F =
    x F.
  \end{align*}
  The universal property of $\Ao(\nabla,F)$ yields a
  morphism $\theta^{(k)} \colon \Ao(\nabla,F) \to B\rtimes
  \Gamma$ such that $\theta^{(k)}_{n}(v)=x$. Using Lemma
  \ref{lemma:s-group}, one easily verifies that the family
  $(\theta^{(k)})_{k}$ is a scaling character group.
\end{proof}
Assume that $B$ carries an involution which is preserved by $\Gamma$.  We call a character group
$(\theta^{(k)})_{k}$ on a $(B,\Gamma)$-Hopf $*$-algebroid
\emph{imaginary} if $\theta^{(k)} \circ * = * \circ
\theta^{(-k)}$ for all $k \in \integers$.
\begin{proposition}
  Let $F \in \GLn(B)$ be $\nabla$-even.  Then
  $\Au(\nabla,F)$ has intertwiners $u \xrightarrow{F^{-1}}
  S^{2}_{n}(u)$ and $\bar u \xrightarrow{(\nabla F
    \nabla^{-1})^{\top}}S^{2}_{n}(\bar u)$, and an
  imaginary scaling character group $\theta$ such that
  $\theta^{(k)}_{n}(u) = \nabla F^{-k}$ and
  $\theta^{(k)}_{n}(\bar u) = F^{k\top}\nabla^{-1}$ for all
  $k\in \integers$.
\end{proposition}
\begin{proof}
  By Lemma \ref{lemma:s-group} i), the first intertwiner is
  the inverse of $S_{n}^{2}(u)=(u^{-\top})^{-\top}=\bar
  u^{-\top}\xrightarrow{F} u$, and the second intertwiner is
  the inverse of $(\bar u^{-\top} \xrightarrow{F}
  u)^{-\top}$.  To construct $\theta$, let $k\in \integers$
  and $x= \nabla F^{-k}$, $y=F^{k\top}\nabla^{-1}$. Using Lemma
  \ref{lemma:rn-grading}, we find
  \begin{align*}
   y&=x^{-\top}, & y^{-\top}&=x, &(\nabla F \nabla^{-1})
  y^{-\top} &= (\nabla F \nabla^{-1}) x = \nabla F^{1-k} = x
  F.
  \end{align*}
  The universal property of the algebra $\Aup(\nabla,F)$ and
  Theorem \ref{theorem:au-prime-iso} yield a morphism
  $\theta^{(k)}\colon \Au(\nabla,F) \to B\rtimes \Gamma$
  such that $\theta^{(k)}_{n}(u)=x$ and
  $\theta_{n}^{(k)}(\bar u)=y$.  Using Lemma
  \ref{lemma:s-group}, one easily verifies that the family
  $(\theta^{(k)})_{k}$ is a scaling character group.  It is
  imaginary because by Lemma \ref{lemma:rn-grading},
 \begin{align*}
   \overline{\theta^{(-k)}_{n}(u)} &= \overline{\nabla F^{k}}
   = \overline{F}^{k} \nabla^{-1} = F^{\top k} \nabla^{-1} =
   \theta^{(k)}_{n}(\bar u) \quad \text{for all }  k\in
   \integers. \qedhere
 \end{align*}
\end{proof}


The case $\Ao(\nabla,F,G)$ requires some preparation.  Let
$F,G \in \GLn(B)$ be $\nabla$-odd and
\begin{align*}
  H&=(\nabla F
  \nabla)^{\top}F^{-1}=\nabla^{-1}F^{\top}\nabla^{-1}F^{-1},
  & Q&=G\nabla \bar G \nabla
\end{align*}
 as before. We say that
a diagram with arrows labeled by matrices commutes if for
all possible directed paths with the same starting and
ending point in the diagram, the products of the labels
along the arrows coincide.

\begin{lemma}
  In the diagram below, (A) commutes if and only if (D)
  commutes, and (B) commutes if and only (C) commutes:
  \begin{align*}
    \xymatrix@C=70pt@R=25pt@l{ \bullet \ar@{<-}[r]^{\nabla
        F^{-\top} \nabla} \ar@{<-}[d]_{\nabla \overline{G}
        \nabla} \ar@{}[rd]|{\text{(D)}} & \bullet
      \ar@{<-}[r]^{F} \ar@{<-}[d]|{\nabla^{-1} G^{\top}
        \nabla^{-1}} \ar@{}[rd]|{\text{(B)}} & \bullet
      \ar@{<-}[d]^{\nabla \overline{G} \nabla}
      \\
      \bullet \ar@{<-}[r]|{\nabla^{-1} \bar F^{-1}
        \nabla^{-1}} \ar@{<-}[d]_{G} \ar@{}[rd]|{\text{(C)}}
      &
      \bullet \ar@{<-}[r]|{\overline{F}^{\top}} \ar@{<-}[d]|{\overline{G}^{\top}} \ar@{}[rd]|{\text{(A)}} & \bullet \ar@{<-}[d]^{G}\\
      \bullet \ar@{<-}[r]_{\nabla F^{-\top} \nabla} &\bullet
      \ar@{<-}[r]_{F} & \bullet }
  \end{align*}
  If all squares commute, then $HQ=QH$, $\overline{G} \nabla
  H^{-1} = \overline{H} \overline{G}  \nabla$, and
 $Q F = F \nabla Q^{\top }\nabla^{-1}$.
\end{lemma}
\begin{proof}
  Applying the transformation $X\mapsto X^{-\top}$ and
  reversing invertible arrows, one can obtain (D) from (A)
  and (C) from (B). If all small squares commute, then the
  three asserted relations follow from the commutativity of the
  large square, of the lower two squares, and of the left
  two squares, respectively.
\end{proof}
\begin{proposition}
  Let $F,G \in \GLn(B)$ be $\nabla$-odd.  Assume
that  $FG^{*}=GF^{*}$ and $F^{*}(\nabla \bar G \nabla)^{*} =
  (\nabla \bar G \nabla)F$, and let $H=
  \nabla^{-1}F^{\top}\nabla^{-1} F^{-1}$.  Then
  $\Ao(\nabla,F,G)$ has an intertwiner $v \xrightarrow{H}
  S^{2}(v)$ and an imaginary scaling character group
  $(\theta^{(k)})_{k}$ such that $\theta^{(k)}_{n}(v) =
  \nabla H^{k}$.
\end{proposition}
\begin{proof}
  We can re-use the arguments in the proof of Proposition
  \ref{proposition:s-ao} and only have to show additionally
  that $\nabla H^{k} \xrightarrow{Q} \nabla H^{k}$ is an
  intertwiner and that $\theta^{(-1)}_{n}(\bar v) =
  \overline{\theta^{(1)}_{n}(v)}$.  But by the lemma above,
  $(\nabla Q \nabla^{-1}) \nabla H^{k} = \nabla Q H^{k} =
  \nabla H^{k} Q$ and
  \begin{align*}
    \theta^{(-1)}_{n}(\bar v) &= \theta^{(1)}_{n}(\overline{G} v
    (\nabla \overline{G} \nabla)^{-1}) = \overline{G} \nabla H^{-1}
    \nabla^{-1} \overline{G}^{-1} \nabla^{-1} =  \overline{H}
    \nabla^{-1} = \overline{\theta^{(1)}_{n}(v)}. \qedhere
  \end{align*}
\end{proof}
\begin{remark}
  Applying the functor \eqref{eq:rn-star} to $v^{-\top}
  \xrightarrow{F} v$, $\bar v \xrightarrow{G} v$, $v
  \xrightarrow{\nabla \bar G \nabla} \bar v$, we obtain
  intertwiners $\bar v^{-\top} \xrightarrow{F^{*}}
  \overline{v}$, $\bar v^{-\top} \xrightarrow{G^{*}}
  v^{-\top}$, $v^{-\top} \xrightarrow{(\nabla \bar G
    \nabla)^{*}} \bar v^{-\top}$, and the conditions
  $FG^{*}=GF^{*}$ and $F^{*}(\nabla \bar G \nabla)^{*} =
  (\nabla \bar G \nabla)F$  amount to commutativity of
  the squares
  \begin{align*}
    \xymatrix@R=15pt@C=20pt{\bar v^{-\top} \ar[r]^{F^{*}}
      \ar[d]_{G^{*}} & \bar v \ar[d]^{G} \\
      v^{-\top} \ar[r]_{F} & v} \qquad\text{ and }\qquad
    \xymatrix@R=15pt@C=20pt{v^{-\top}
      \ar[r]^{F} \ar[d]_{(\nabla \bar G \nabla)^{*}}
      & v \ar[d]^{\nabla \bar G \nabla} \\ \bar v^{-\top}
      \ar[r]_{F^{*}} & \bar v.}
  \end{align*}
  If $Q=G \nabla \bar G \nabla$ is scalar, then both
  conditions evidently are equivalent.
\end{remark}

\subsection{The full dynamical quantum group $\SUd$}

\label{section:sud}

In \cite{koelink:su2}, Koelink and Rosengren studied a
dynamical quantum group $\FR(\SU)$ that arises from a
dynamical $R$-matrix via the generalized FRT-construction of
Etingof and Varchenko.  We first recall its definition, then
show that this dynamical quantum group coincides with
$\Ao(\nabla,F,G)$ for specific choice of
$B,\Gamma,\nabla,F,G$, and finally construct a refinement
that includes several interesting limit cases.

We shall slightly reformulate the definition of
$\FR(\SL)$ and $\FR(\SU)$ given in \cite[\S
2.2]{koelink:su2} so that it fits better
with our approach.

Fix $q\in (0,1)$. Let $\MC$ be the algebra of
meromorphic functions on the plane and let
$\integers$ act on $B$ such that $b \stackrel{k}{\mapsto}$
$b_{(k)}:=b(\frei - k)$ for all $b\in B$, $k \in \integers$.
Define $f \in \MC$ by
\begin{align} \label{eq:k-f}
  f(\lambda) &=
  q^{-1}\frac{q^{2(\lambda+1)}-q^{-2}}{q^{2(\lambda+1)}-1} =
  \frac{q^{\lambda+2} -
    q^{-(\lambda+2)}}{q^{\lambda+1}-q^{-(\lambda-1)}} \quad
  \text{for all } \lambda\in \complex.
\end{align}
Then the $(\MC,\integers)$-Hopf algebroid $\FR(\SL)$ is the
universal $(\MC,\integers)$-algebra with generators
$\alpha,\beta,\gamma,\delta$ satisfying
\begin{gather}
\label{eq:k-0}
\begin{aligned}
\partial_{\alpha} &= (1,1), & \partial_{\beta} &= (1,-1), &
\partial_{\gamma} &=(-1,1), &\partial_{\delta} &= (-1,-1),
\end{aligned} \\
 \label{eq:k-1} 
  \begin{aligned}
    \alpha\beta &= s(f_{(1)})\beta \alpha, & \alpha\gamma &=
    r(f)\gamma\alpha, & \beta\delta &= r(f)\delta\beta, &
    \gamma\delta &= s(f_{(1)})\delta\gamma, 
  \end{aligned}  \\
  \label{eq:k-2}
    \frac{r(f)}{s(f)} \delta\alpha - \frac{1}{s(f)}\beta\gamma =
    \alpha\delta - r(f)\gamma\beta  = \frac{r(f_{(1)})}{s(f_{(1)})}
    \alpha\delta - r(f_{(1)})\beta\gamma = \delta\alpha -
    \frac{1}{s(f_{(1)})} \gamma\beta = 1,
\end{gather}
and with comultiplication, counit and antipode given by
\begin{gather} \label{eq:k-d}
  \begin{aligned}
    \Delta(\alpha) &= \alpha \todot \alpha + \beta\todot
    \gamma, & \Delta(\beta) &= \alpha \todot \beta + \beta
    \todot \delta, \\
    \Delta(\gamma) &= \gamma \todot \alpha + \delta \todot
    \gamma, & \Delta(\delta) &= \gamma \todot \beta + \delta
    \todot \delta,
  \end{aligned} \\ \label{eq:k-e}
  \begin{aligned}
    \epsilon(\alpha) &= \partial^{r}_{\alpha} =
 \partial^{s}_{\alpha}, & \epsilon(\beta) =
    \epsilon(\gamma) &= 0, & \epsilon(\delta) &= \partial^{r}_{\delta} =
     \partial^{s}_{\delta},
  \end{aligned} \\ \label{eq:k-s}
  \begin{aligned}
    S(\alpha) &= \frac{r(f)}{s(f)} \delta, & S(\beta) &=
    -\frac{1}{s(f)}\beta, & S(\gamma) &= -r(f)\gamma, &
    S(\delta) &= \alpha.
  \end{aligned}
\end{gather}
Equip $\MC$  with the involution
given by $b^{*}(\lambda) = \overline{b(\overline{\lambda})}$
for all $b\in \MC$, $\lambda \in
\complex$. Then $\FR(\SL)$ can be equipped with an
involution such that
\begin{align}
  \label{eq:k-star}
\alpha^{*} &= \delta, & \beta^{*} &= -q\gamma, & \gamma^{*}
&= -q^{-1}\beta, & \delta^{*}&= \alpha,
\end{align}
and one obtains a $(\MC,\integers)$-Hopf
$*$-algebroid which is denoted by $\FR(\SU)$
\cite{koelink:su2}.
\begin{proposition} \label{proposition:sud}
  Let  $\nabla =
  \diag(1,-1)$, $F =
  \begin{pmatrix}
    0 & -1 \\ f_{(1)}^{-1} & 0
  \end{pmatrix}$, $G=
 \begin{pmatrix}
   0 & -1 \\ q^{-1} & 0
 \end{pmatrix}$. Then there exist isomorphisms of
 $(\MC,\integers)$-Hopf (*-)algebroids
 $A^{\MC}_{\mathrm{o}}(\nabla,F) \to \FR(\SL)$ and
 $A^{\MC}_{\mathrm{o}}(\nabla,F,G) \to \FR(\SU)$ whose
 extensions to matrices map $v$ to $
 \begin{pmatrix}
   \alpha & \beta \\ \gamma & \delta
 \end{pmatrix}
$.
\end{proposition}
\begin{proof}
  First, note that the function $\lambda \mapsto
  q^{\lambda}$ and hence also $f$ is self-adjoint, and that
  \begin{align*}
    \hat F:= \nabla F \nabla &=
    \begin{pmatrix}
      0 & -1 \\
      f^{-1} & 0
    \end{pmatrix}, & \hat G:=\nabla G \nabla &= G,
     & FG^{*} &=
    \begin{pmatrix}
      1 & 0 \\ 0 & (qf_{(1)})^{-1}
    \end{pmatrix} = GF^{*}.
  \end{align*}
  Therefore, $A:=A_{\mathrm{o}}^{\MC}(\nabla,F)$ and
  $A_{\mathrm{o}}^{\MC}(\nabla,F,G)$ are well-defined.
 Since $\nabla G \nabla \overline{G} = G^{2}
  =q^{-1} \in M_{2}(\MC)$, the latter algebra coincides with
  the former.
  
  Write $v\in M_{2}(A)$ as $v=
\begin{pmatrix}
  \alpha' & \beta' \\ \gamma' & \delta'
\end{pmatrix}
$ and write \eqref{eq:k-0}$'$--\eqref{eq:k-star}$'$ for the
relations \eqref{eq:k-0}--\eqref{eq:k-star} with
$\alpha',\beta',\gamma',\delta'$ instead of
$\alpha,\beta,\gamma,\delta$.  Then the relation
$\partial_{v} = \nabla$ is equivalent to
\eqref{eq:k-0}$'$. The relation $v^{-\top} = r_{2}(\hat
F^{-1})vs_{2}(F)$ is equivalent to
\begin{align*}
v^{-1}&=
\left(\begin{pmatrix}
    0 & r(f) \\ -1 & 0
  \end{pmatrix}
  \begin{pmatrix}
    \alpha' & \beta' \\ \gamma' & \delta'
  \end{pmatrix}
  \begin{pmatrix}
    0 & -1 \\ s(f_{(1)}^{-1}) & 0
  \end{pmatrix} \right)^{\top}
 = \begin{pmatrix}
      \frac{r(f)}{s(f)} \delta' & - \frac{1}{s(f)}\beta' \\
      -r(f)\gamma' & \alpha'
    \end{pmatrix},
\end{align*}
and multiplying out $v^{-1}v=1=vv^{-1}$ and using
\eqref{eq:k-0}$'$, we find that this relation is equivalent
to \eqref{eq:k-1}$'$ and \eqref{eq:k-2}$'$.  Hence, there
exists an isomorphism of $(\MC,\integers)$-algebras
$A \to \FR(\SL)$ sending
$\alpha',\beta',\gamma',\delta'$ to
$\alpha,\beta,\gamma,\delta$. This isomorphism is compatible
with the involution, comultiplication, counit and antipode
because \eqref{eq:k-d}$'$--\eqref{eq:k-star}$'$ are
equivalent to $\Delta_{2}(v)=v\mal v$,
$\epsilon_{2}(v)=\partial_{v}$, $S_{2}(v)=v^{-1}$ and
\begin{align*}
\bar v&=r_{2}(\hat G^{-1})vs_{2}(G) =
\begin{pmatrix}
  0 & q \\ -1 & 0
\end{pmatrix}
\begin{pmatrix}
  \alpha' & \beta' \\ \gamma' & \delta'
\end{pmatrix}
 \begin{pmatrix}
   0 & -1 \\ q^{-1} & 0
 \end{pmatrix}
 =
 \begin{pmatrix}
   \delta' & -q\gamma' \\ -q^{-1}\beta' & \alpha'
 \end{pmatrix}. \qedhere
\end{align*}
\end{proof}
We now refine the definition above as follows. The first
idea is to replace the base $\MC$ by the
$\integers$-invariant subalgebra containing $f$ and
$f^{-1}$. This subalgebra can be described in terms of the
functions $x(\lambda)=q^{\lambda}$,
$y(\lambda)=q^{-\lambda}$ and $z=x-y$ as follows. Since
$f=z_{(-2)}/z_{(-1)}$, this subalgebra is generated by all
fractions $z_{(k)}/z_{(l)}$, where $k,l \in \integers$, and
since $z_{(-1)}-qz=(q^{-1}-q)q^{-\lambda}$, also by all
fractions $x/z_{(k)}$ and $y/z_{(k)}$, where $k\in
\integers$. The second idea is to drop the relation $xy=1$
to allow the limit cases $\lambda \to \pm \infty$, and
regard $x,y$ as canonical coordinates on $\complex
P^{1}$. Finally, we also regard $q$ as a variable.

Let us now turn to the details.
Denote by $R \subset \complex(Q)$ the localization of
$\complex[Q]$ with respect to $Q$ and the polynomials
\begin{align*}
  S_{k} = (1-Q^{2k})/(1-Q^{2}) = 1 + Q^{2} + \cdots +
  Q^{2(k-1)}, \quad \text{where } k \in
  \naturals.
\end{align*}
Let $\integers$ act on the algebra $\complex(Q,X,Y)$
of rational functions in $Q,X,Y$ by
\begin{align*}
  Q_{(k)} &= Q, & X_{(k)} &= Q^{-k}X, & Y_{(k)}&= Q^{k}Y
  &&\text{for all } k\in \integers,
\end{align*}
where the lower index $(k)$ denotes the action of $k$.
Denote by $B \subset \complex(Q,X,Y)$  the subalgebra
generated by $R$ and all elements
\begin{align*}
  Z_{k,l} := (X-Y)_{(k)} / (X-Y)_{(l)}, \quad \text{where }
  k,l \in \integers.
\end{align*}
We equip $B$ with the induced action of $\integers$ and the
involution given by $Q=Q^{*}$ and $Z_{k,l}^{*} = Z_{k,l}$
for all $k,l\in \integers$. Note that this involution is the one
inherited from $\complex(Q,X,Y)$ when $Q=Q^{*}$ and either
$X^{*}=X$, $Y^{*}=Y$ or $X^{*}=-X$, $Y^{*}=-Y$.
Finally, let
\begin{align*}
  \nabla&= (1,-1), &
  F&=
  \begin{pmatrix}
    0 &  -1 \\ Z_{0,-1} & 0
  \end{pmatrix}, &
  G&=
  \begin{pmatrix}
      0 & -Q \\ 1 & 0
    \end{pmatrix}.
\end{align*}
Then $FG^{*}=G^{*}F$ and $G\nabla \bar G \nabla = G^{2} =
\diag(-Q,-Q)$.  
\begin{definition} \label{definition:sud} We let
  $\mathcal{O}(\SUd):=\Ao(\nabla,F,G)$.
\end{definition}
Thus, $\mathcal{O}(\SUd)$ is generated by the entries
$\alpha,\beta,\gamma,\delta$ of a $2\times 2$-matrix $v$
which satisfy the relations
\eqref{eq:k-0}--\eqref{eq:k-star} with $Z_{-2,-1}$ and $Q$
instead of $f$ and $q$.  This $(B,\integers)$-Hopf
$*$-algebroid aggregates several other interesting quantum
groups and quantum groupoids which can be obtained by
suitable base changes as follows.

Denote by $z\in \MC$ the function $\lambda \mapsto
q^{\lambda}-q^{-\lambda}$.  Equip $\complex(\lambda)$ with
an involution such that $\lambda^{*}=\lambda$, and a
$\integers$-action such that $\lambda_{(k)}=\lambda-k$. Let
$\Omega=(0,1]\times [-\infty,\infty]$ and let $\integers$
act on $C(\Omega)$ by $g_{(k)}(q,t) = g(q,t-k)$ for all
$g\in C(\Omega)$, $(q,t)\in \Omega$, $k\in \integers$.
\begin{lemma}
  There exist $\integers$-equivariant $*$-homomorphisms
  \begin{gather*}
   \begin{aligned}
     \mathrm{i)} & \!\! & \pi^{q}_{\MC} &\colon B \to \MC,
     &\!\! Q&\mapsto q, &\!\! Z_{k,l} &\mapsto
     \frac{z_{(k)}}{z_{(l)}}, \qquad
     \text{for  } q\in (0,1) \cup (1,\infty),\\
     \mathrm{ii)} & \!\! & \pi^{1}_{\MC} &\colon B \to
     \complex(\lambda), &\!\! Q &\mapsto 1, &\!\! Z_{k,l} &\mapsto
     \frac{\lambda
       -k}{\lambda-l}, \\
     \mathrm{iii)} &\!\! & \pi_{\pm \infty} &\colon B \to R, &\!\! Q
     &\mapsto Q, &\!\! Z_{k,l} &\mapsto \frac{Q^{\pm k}}{Q^{\pm
         l}} = Q^{\pm  k \mp l}, \\
     \mathrm{iv)} & \!\! & \pi^{q}_{\pm \infty} &\colon B\to
     \complex, &\!\! Q &\mapsto q, &\!\! Z_{k,l} &\mapsto q^{\pm k
       \mp l}, \qquad
     \text{for } q \in (0,\infty), \\
     \mathrm{v)} & \!\! & \pi_{\Omega} &\colon B\to C(\Omega), &\!\!
     Q &\mapsto\left((q,t) \mapsto q\right), &\!\! Z_{k,l}
     &\mapsto \left((q,t) \mapsto
     \begin{cases}
     \tfrac{q^{t-k}+q^{k-t}}{q^{t-l}+q^{l-t}}, & t\in \reals, \\      q^{\pm
       k \mp l}, & t =\pm \infty
     \end{cases}\right).
    \end{aligned}
  \end{gather*}
\end{lemma}
\begin{proof}
  i) Restrict the homomorphism $\pi\colon
  \complex(Q,X,Y)\to\MC$ given by $Q\mapsto q$,
  $X\mapsto (\lambda \mapsto q^{\lambda})$, $Y \mapsto
  (\lambda \mapsto q^{-\lambda})$ to $B$.

  ii) Use i) and the fact that for all $k,l\in \integers$
  and $\lambda\in \complex \setminus \{l\}$,
  \begin{align*}
    \lim_{q\to 1}
    \pi^{q}_{\MC}(Z_{k,l})(\lambda)
    = \lim_{q\to
      1}\frac{q^{\lambda-k}-q^{k-\lambda}}{q^{\lambda-l}-q^{l-\lambda}}
    = \frac{\lambda-k}{\lambda-l}.
  \end{align*}

  iii) Define $\pi\colon \complex[Q,X,Y] \to R$ by $Q\mapsto
  Q$, $X\mapsto 1$, $Y \mapsto 0$.  Then $\pi$ extends to
  the localization $B$ of $\complex[Q,X,Y]$, giving
  $\pi_{-\infty}$, because $\pi((X-Y)_{(k)})= Q^{-k}$ is
  invertible for all $k\in \integers$. This homomorphism $\pi_{-\infty}$
  evidently is involutive, and $\integers$-equivariant because
  $\pi_{-\infty}(Z_{k+j,l+j})
  = \pi_{-\infty}(Z_{k,l})$ for all $j\in \integers$.
  Similarly, one obtains $\pi_{+\infty}$.

  iv) Immediate from iii).

  v) Define $\pi\colon \complex[Q,X,Y] \to C((0,1]\times
  \reals)$ by $Q\mapsto ((q,t) \mapsto q)$, $X \mapsto
  ((q,t) \mapsto iq^{t})$, $Y \mapsto ((q,t) \mapsto
  -iq^{t})$. Since $\pi((X-Y)_{(k)}) = i(q^{t-k}+q^{k-t})$
  is invertible for all $t\in \reals$, $k\in \integers$,
  this $\pi$ extends to $B$. Moreover, each $\pi(Z_{k,l})$
  extends to a continuous function on $C(\Omega)$ as
  desired, giving $\pi_{\Omega}$. 
\end{proof}
Note that $\pi_{+\infty}^{1}=\pi_{-\infty}^{1}$. Using this
map, we obtain for each algebra $C$ with an action
by $\integers$ an $\integers$-equivariant homomorphism
$\pi^{1}_{C}\colon B\to C$ sending $Q$ and each $Z_{k,l}$
to $1_{C}$.
\begin{proposition} \label{proposition:sud-base-change}
There exist isomorphisms of Hopf $*$-algebroids as follows:
  \begin{enumerate}
  \item $(\pi^{q}_{\MC})_{*} \mathcal{O}(\SUd) \cong
    \mathcal{F}_{R}(\SU)$ for each $q\in (0,1) \cup
    (1,\infty)$;
  \item $(\pi^{q}_{-\infty})_{*} \mathcal{O}(\SUd) \cong
    \mathcal{O}(\SUq)$  for each $q\in (0,\infty)$;
  \item  $(\pi^{q}_{\infty})_{*}  \mathcal{O}(\SUd) \cong
    \mathcal{O}(\SUq)^{\op}$ for each $q\in (0,\infty)$;
  \item $(\pi^{1}_{\complex[X]})_{*}\mathcal{O}(\SUd)$ is
    isomorphic to the $(\complex[X],\integers)$-Hopf
    $*$-algebroid in Example \ref{example:ao-su}.
  \end{enumerate}
\end{proposition}
\begin{proof}
   i) This is immediate from the definitions and Proposition
     \ref{proposition:sud}.

  ii), iii) Let $\pi^{\pm}=\pi^{q}_{\pm\infty}$. Then
  $(\pi^{\pm})_{*}\mathcal{O}(\SUd)$ is generated by the entries
  $\alpha',\beta',\gamma',\delta$ of a matrix $v'$ such that
  $\beta' = -q\gamma'{}^{*}$ and
  $\delta'=\alpha'{}^{*}$. Moreover, $v'{}^{-\top}
  =\pi^{\pm}_{2}(F)^{-1}v'\pi^{\pm}_{2}(F)$ and $\bar v' =
  \pi^{\pm}_{2}(G)^{-1}v'\pi^{\pm}_{2}(G)$, where
  \begin{align*}
    \pi^{-}_{2}(F) &=
    \begin{pmatrix}
      0 & -1 \\ \pi(Z_{0,-1}) & 0
    \end{pmatrix}
    =    \begin{pmatrix}
      0 & -1 \\ q^{-1} & 0
    \end{pmatrix} =\pi^{\pm}_{2}(G), &
    \pi^{+}_{2}(F) &=
    \begin{pmatrix}
      0 & -1 \\ q & 0
    \end{pmatrix}.    
  \end{align*}
  In the case of $\pi_{-}$, we find that $v'$ is unitary,
  and obtain the usual presentation of
  $\mathcal{O}(\SUq)$. Multiplying out the relation
  $v'{}^{-\top} =\pi^{+}_{2}(F)^{-1}v'\pi^{+}_{2}(F)$, one
  easily verifies the assertion on $\pi^{+}$.

  iv) Immediate from the relations
  $(\pi^{1}_{\complex[X]})_{2}(F)=(\pi^{1}_{\complex[X]})_{2}(G)
  =
  \begin{pmatrix}
    0 & -1 \\  1 & 0
  \end{pmatrix}$.
\end{proof}
We expect most of the results of \cite{koelink:su2} to carry
over from $\FR(\SU)$ to $\mathcal{O}(\SUd)$.

\section{The level of
  universal $C^{*}$-algebras}

\label{section:universal}

Throughout this section, we shall only work with unital
$C^{*}$-algebras. We assume all $*$-homomorphisms to be
unital, and $B$ to be a commutative, unital $C^{*}$-algebra
equipped with a left action of a discrete group $\Gamma$.
Given a subset $X$ of a normed space $V$, we denote by
$\overline{X} \subseteq v$ its closure and by
$[X]\subseteq V$ the closed linear span of $X$.

\subsection{The maximal cotensor product of $C^{*}$-algebras
with respect to $C^{*}(\Gamma)$}

\label{section:c}
This subsection reviews the cotensor product of
$C^{*}$-algebras with respect to the Hopf $C^{*}$-algebra
$C^{*}(\Gamma)$ and develops the main properties that will
be needed in \S\ref{section:cb}.  The material presented
here is certainly well known to the experts, but we didn't
find a suitable reference.

We first recall a few preliminaries.

Let $A$ be a $*$-algebra. A \emph{representation} of $A$ is
a $*$-homomorphism into a $C^{*}$-algebra. Such a
representation $\pi$ is \emph{universal} if every other
representation of $A$ factorizes uniquely through $\pi$.  A
universal representation exists if and only if for each
$a\in A$,
\begin{align*}
  |a|:= \sup\{ \|\pi(a)\| : \pi\text{ is a $*$-homomorphism
    of $A$ into some $C^{*}$-algebra}\} < \infty.
\end{align*}
Indeed, if $|a|$ is finite for all $a \in A$, then the
separated completion of $A$ with respect to $|\!-\!|$
carries a natural structure of a $C^{*}$-algebra, which is
denoted by $C^{*}(A)$ and called the enveloping
  $C^{*}$-algebra of $A$, and the natural representation
$A\to C^{*}(A)$ is universal.

The maximal tensor product of $C^{*}$-algebras $A$ and $C$
is the enveloping $C^{*}$-algebra of the algebraic tensor
product $A\otimes C$, and will be denoted by
$A\maxtimes C$.

The full group $C^{*}$-algebra
$C^{*}(\Gamma)$ of $\Gamma$ is the enveloping
$C^{*}$-algebra of the group algebra $\complex \Gamma$. We
denote by $\Delta_{\Gamma} \colon C^{*}(\Gamma) \to
C^{*}(\Gamma) \maxtimes C^{*}(\Gamma)$ the
\emph{comultiplication}, given by $\gamma\mapsto
\gamma\otimes \gamma$ for all $\gamma\in \Gamma$, and by
$\epsilon_{\Gamma} \colon C^{*}(\Gamma) \to \complex$ the
\emph{counit}, given by $\gamma \mapsto 1$ for all
$\gamma\in \Gamma$. Clearly, $(\epsilon_{\Gamma}\maxtimes
\Id) \Delta_{\Gamma} =\Id = (\Id\maxtimes
\epsilon_{\Gamma})\Delta_{\Gamma}$.

A completely positive (contractive) map, or brielfy
c.p.(c.)-map, from a $C^{*}$-algebra $A$ to a $C^{*}$-algebra
$C$ is a linear map $\phi \colon A\to C$ such that
$\phi_{n}\colon \Mn(A) \to \Mn(C)$ is positive (and
 $\|\phi_{n}\|\leq 1$) for all $n \in \naturals$.
\begin{definition} \label{definition:c-cg}
  A \emph{$(\complex,\Gamma)$-$C^{*}$-algebra} is a unital
  $C^{*}$-algebra $A$ with injective unital $*$-homomor\-phisms
  $\delta_{A}\colon A \to \cg \maxtimes A$ and $\bar \delta_{A}
  \colon A \to A\maxtimes \cg$ such that 
  $(\Id \maxtimes \delta_{A}) \circ \delta_{A}  = (\Delta_{\Gamma}
    \maxtimes \Id)  \circ \delta_{A}$,
    $(\bar \delta_{A} \maxtimes \Id) \circ \bar \delta_{A} = (\Id
    \maxtimes \Delta_{\Gamma}) \circ \bar \delta_{A}$ and 
    $(\delta_{A} \maxtimes \Id) \circ \bar \delta_{A} = (\Id
    \maxtimes \bar \delta_{A}) \circ \delta_{A}$.
  A \emph{morphism} of $(\complex,\Gamma)$-$C^{*}$-algebras
  $A$ and $C$ is a unital $*$-homomorphism $\pi \colon A\to
  C$ satisfying $\delta_{C} \circ \pi = (\Id \maxtimes \pi)
  \circ \delta_{A}$ and $\bar \delta_{C} \circ \pi = (\pi
  \maxtimes \Id) \circ \bar \delta_{A}$. We denote by
  $\cgcalg$ the category of all
  $(\complex,\Gamma)$-$C^{*}$-algebras. Replacing
  $*$-homomorphisms by c.p.-maps, we define c.p.-maps of
  $(\complex,\Gamma)$-$C^{*}$-algebras and the category
  $\cgcalg^{\cp}$.
\end{definition}
\begin{remark}
  Let $A$ be a $(\complex,\Gamma)$-$C^{*}$-algebra. Then
  $(\epsilon_{\Gamma} \maxtimes \Id) \circ \delta_{A} = \Id_{A}$
because
\begin{align*}
  \delta_{A} (\epsilon_{\Gamma} \maxtimes \Id)\circ
  \delta_{A} = (\epsilon_{\Gamma} \maxtimes \Id\maxtimes
  \Id) \circ (\Id \maxtimes \delta_{A})\circ\delta_{A} =
  ((\epsilon_{\Gamma} \maxtimes \Id) \circ \Delta_{\Gamma}
  \maxtimes \Id) \circ \delta_{A} = \delta_{A},
\end{align*}
and likewise $(\Id \maxtimes \epsilon_{\Gamma})\circ
\bar\delta_{A}=\Id_{A}$.

\end{remark}
Let $A$ and $C$ be $(\complex,\Gamma)$-$C^{*}$-algebras. 
Then 
 the maximal tensor product $A\maxtimes C$ is a
$(\complex,\Gamma)$-$C^{*}$-algebra with respect to $\delta_{A}
\maxtimes \Id$ and $\Id \maxtimes \bar \delta_{C}$, and the
assignments $(A,C) \mapsto A\maxtimes C$ and $(\phi,\psi)
\mapsto \phi\maxtimes \psi$
define a product $-\maxtimes-$ on $\cgcalg^{(\cp)}$ that is
associative in the obvious sense.  Unless $\Gamma$ is
trivial, this product can not be unital because it forgets
$\bar\delta_{A}$ and $\delta_{C}$. 

With respect to the
restrictions of $\delta_{A}\maxtimes \Id$ and
$\Id\maxtimes\bar \delta_{C}$, 
 the subspace
\begin{align*}
  A \gmaxtimes C := \{ x \in A \maxtimes C : (\bar
  \delta_{A} \maxtimes \Id)(x) = (\Id \maxtimes
  \delta_{C})(x) \} \subseteq A \maxtimes C
\end{align*}
evidently is a $(\complex,\Gamma)$-$C^{*}$-algebra
again. Moreover, given morphisms of
$(\complex,\Gamma)$-$C^{*}$-algebras $\phi\colon A\to C$ and
$\psi \colon D \to E$, the product $\phi\maxtimes \psi$
restricts to a morphism $\phi \gmaxtimes \psi \colon
A\gmaxtimes D \to C\gmaxtimes E$. We thus obtain a second
product $-\gmaxtimes-$ on $\cgcalg^{(\cp)}$ that is
associative in the natural sense, and unital in the
following sense. 

Regard $\cg$ as a $(\complex,\Gamma)$-$C^{*}$-algebra with
respect to $\Delta_{\Gamma}$.  Then for each
$(\complex,\Gamma)$-$C^{*}$-algebra $A$, the maps
$\delta_{A}$ and $\bar\delta_{A}$ are isomorphisms of $(\complex,\Gamma)$-$C^{*}$-algebras
\begin{align*}
  \delta_{A} &\colon A \xrightarrow{\cong} \cg\gmaxtimes A,&
  \bar\delta_{A} & \colon A \xrightarrow{\cong} A\gmaxtimes
  \cg.
\end{align*}
Indeed, they evidently are morphisms, and surjective because 
\begin{align*}
  x &=
  (\epsilon_{\Gamma} \maxtimes \Id \maxtimes
  \Id)((\Delta_{\Gamma} \maxtimes \Id)(x)) = (\epsilon_{\Gamma}
  \maxtimes \Id \maxtimes \Id) ((\Id \maxtimes \delta_{A})(x)) =
  \delta_{A}((\epsilon_{\Gamma} \maxtimes \Id)(x))
\end{align*}
for each $x\in C^{*}(\Gamma) \gmaxtimes A$ and likewise $y =
\bar \delta_{A}((\Id\maxtimes \epsilon_{\Gamma})(y))$ for
each $y\in A\gmaxtimes C^{*}(\Gamma)$.

We next construct a natural transformation $p \colon
(-\maxtimes -) \to (-\gmaxtimes -)$ which will be needed to
prove associativity of the product of
$(B,\Gamma)$-$C^{*}$-algebras in \S\ref{section:cb}. The
construction is based on ideas taken from
\cite[\S7]{baaj:1}, and carries over from $C^{*}(\Gamma)$ to
any Hopf $C^{*}$-algebra $H$ equipped with a Haar mean
$H\maxtimes H\to H$; see also
\cite{maghfoul}.
\begin{lemma}
  There exists a unique state $\nu$ on $C^{*}(\Gamma)
  \maxtimes C^{*}(\Gamma)$ such that
  $\nu(\gamma\otimes\gamma')=\delta_{\gamma,\gamma'}1$ for
  all $\gamma,\gamma' \in \Gamma$. Moreover, $\nu\circ
  \Delta_{\Gamma} = \epsilon_{\Gamma}$ and $(\Id \maxtimes
  \nu) \circ (\Delta_{\Gamma}
  \maxtimes \Id)=(\nu\maxtimes \Id) \circ  (\Id \maxtimes
  \Delta_{\Gamma})$.
\end{lemma}
\begin{proof}
  This follows from \cite[Theorem 0.1]{maghfoul}, but let us
  include the short direct proof. Uniqueness is clear. To
  construct $\nu$, denote by $(\epsilon_{\gamma})_{\gamma\in
    \Gamma}$ the canonical orthonormal basis of
  $l^{2}(\Gamma)$, by $\lambda,\rho \colon C^{*}(\Gamma) \to
  \mathcal{L}(l^{2}(\Gamma))$ the representations given by
  $\lambda(\gamma)\epsilon_{\gamma'}
  =\epsilon_{\gamma\gamma'}$ and
  $\rho(\gamma)\epsilon_{\gamma'} =
  \epsilon_{\gamma'\gamma^{-1}}$ for all $\gamma,\gamma'\in
  \Gamma$, and by $\lambda \times \rho \colon C^{*}(\Gamma)
  \maxtimes C^{*}(\Gamma) \to \mathcal{L}(l^{2}(\Gamma))$
  the representation given by $x \otimes y \mapsto
  \lambda(x)\rho(y)$. Then $\nu:=\langle
  \epsilon_{e}|(\lambda\times \rho)(-)\epsilon_{e}\rangle$
  satisfies $\nu(\gamma\otimes
  \gamma')=\delta_{\gamma,\gamma'}1$ for all
  $\gamma,\gamma'\in \Gamma$. The remaining equations follow
  easily.
\end{proof}
\begin{lemma}
  \begin{enumerate}
  \item For every $(\complex,\Gamma)$-$C^{*}$-algebra $A$,
the maps
\begin{align*}
  \bar p_{A}&:= (\Id \maxtimes \nu) ( \bar \delta_{A}
  \maxtimes \Id) \colon A \maxtimes C^{*}(\Gamma) \to A,
  & p_{A} &:= (\nu \maxtimes \Id) ( \Id \maxtimes
  \delta_{A}) \colon C^{*}(\Gamma) \maxtimes A \to A
\end{align*}
are morphisms in $\cgcalg^{\cp}$ and satisfy $p_{A} \circ
\delta_{A}$ and $\bar p_{A}\circ \bar \delta_{A} = \Id$.
  \item The families $(p_{A})_{A}$ and $(\bar p_{A})_{A}$
    are natural transformations from $-\maxtimes \cg$ and
    $\cg \maxtimes -$, respectively, to $\Id$, regarded as
    functors on $\cgcalg^{\cp}$.
  \end{enumerate}
\end{lemma}
\begin{proof}
  i) The map $p_{A}$ is a morphism in $\cgcalg^{\cp}$
  because $\bar \delta_{A} \circ p_{A} = (p_{A}
  \maxtimes \Id) \circ \bar \delta_{A}$ and
  \begin{align*}
    \delta_{A} \circ p_{A} &= (\nu\maxtimes \Id \maxtimes
    \Id)\circ(\Id \maxtimes \Id \maxtimes
    \delta_{A})(\Id \maxtimes \delta_{A}) \\
    &=(\nu\maxtimes \Id \maxtimes \Id) \circ (\Id \maxtimes
    \Delta_{\Gamma} \maxtimes
    \Id)\circ(\Id \maxtimes \delta_{A}) \\
    &= (\Id \maxtimes \nu \maxtimes \Id)\circ(\Delta_{\Gamma}
    \maxtimes \Id \maxtimes \Id)\circ(\Id \maxtimes \delta_{A})
    \\
    &=(\Id \maxtimes \nu \maxtimes \Id)\circ(\Id \maxtimes \Id
    \maxtimes \delta_{A})\circ(\Delta_{\Gamma} \maxtimes \Id) =
    (\Id \maxtimes p_{A})\circ(\Delta_{\Gamma} \maxtimes \Id).
  \end{align*}
  Moreover, $p_{A} \circ \delta_{A} = (\nu \maxtimes
  \Id)(\Id \maxtimes \delta_{A}) \delta_{A} = (\nu
  \Delta_{\Gamma} \maxtimes \Id) \delta_{A} =
  (\epsilon_{\Gamma} \maxtimes \Id)\delta_{A} = \Id$ and
  similarly $\bar p_{A} \circ \bar \delta_{A} = \Id$.

  ii) This follows from the fact that $(\delta_{A})_{A}$ and
  $(\bar \delta_{A})_{A}$ are natural
  transformations.
\end{proof}
\begin{proposition}
  \begin{enumerate}
  \item Let $A,C$ be
    $(\complex,\Gamma)$-$C^{*}$-algebras. Then the map
    \begin{align*}
      p_{A,C} := (\Id \maxtimes \nu \maxtimes
      \Id)\circ(\bar \delta_{A} \maxtimes \delta_{C}) \colon
      A\maxtimes C\to A\maxtimes C
    \end{align*}
    is equal to $(\Id\maxtimes p_{C}) \circ (\bar \delta_{A}
    \maxtimes \Id)$ and $(\bar p_{A} \maxtimes \Id)\circ(\Id
    \maxtimes\delta_{C})$, a morphism in $\cgcalg^{\cp}$, and a
    conditional expectation onto $A\gmaxtimes C\subseteq
    A\maxtimes C$ in the sense that $p_{A,C}(xyz) =
    xp_{A,C}(y)z$ for all $x,z \in A\gmaxtimes C$ and $y \in
    A\maxtimes C$.
  \item The family $(p_{A,C})_{A,C}$ is a natural
    transformation from $-\maxtimes -$ to $-\gmaxtimes-$,
    regarded as functors on $\cgcalg^{\cp}\times
    \cgcalg^{\cp}$.
  \end{enumerate}
\end{proposition}
\begin{proof}
  i) The equality follows immediately from the definitions
  and implies that $p_{A,C}$ is a morphism as claimed. Next,
  $p_{A,C}(A\maxtimes C) \subseteq A\gmaxtimes C$ because
\begin{align*}
  (\bar\delta_{A}\maxtimes \Id)\circ p_{A,C} &= (\Id \maxtimes
  \Id \maxtimes \nu \maxtimes \Id)\circ (\bar \delta_{A}
  \maxtimes \Id \maxtimes \Id \maxtimes \Id)\circ(\bar \delta_{A}
  \maxtimes
  \delta_{C}) \\
  &= (\Id \maxtimes \Id \maxtimes \nu \maxtimes \Id)\circ (\Id
  \maxtimes \Delta_{\Gamma} \maxtimes \Id \maxtimes
  \Id)\circ(\bar \delta_{A} \maxtimes
  \delta_{C}) \\
  &= (\Id \maxtimes \nu \maxtimes \Id \maxtimes \Id)\circ (\Id
  \maxtimes \Id \maxtimes \Delta_{\Gamma} \maxtimes
  \Id)\circ(\bar \delta_{A} \maxtimes \delta_{C}) = (\Id
  \maxtimes \delta_{C}) \circ p_{A,C}.
\end{align*}
On the other hand,  $  p_{A,C}(x) = (\bar p_{A}
  \maxtimes \Id) ((\Id \maxtimes
  \delta_{C})(x)) = (\bar p_{A}
  \maxtimes \Id) ((\delta_{A} \maxtimes
  \Id)(x)) = x$ for all $x\in A\gmaxtimes C$.
Thus,  $p_{A,C}$ is a completely positive projection from
$A\maxtimes C$ onto $A\gmaxtimes C$ and hence a conditional
expectation (see, e.g., \cite[Proposition 1.5.7]{brown-ozawa}).

ii) Straightforward.
\end{proof}

Denote by $\cgsalgu \subseteq \cgsalg$ the full subcategory
formed by all $(\complex,\Gamma)$-$*$-algebras that have an
enveloping $C^{*}$-algebra. We shall need an adjoint
pair of functors
\begin{align} \label{eq:c-functors}
  \xymatrix@C=40pt{
    \cgsalgu \ar@<+3pt>[r]^{\bfC} & \cgcalg. \ar@<+3pt>[l]^{\bfU}
  }
\end{align}

The functor $\bfC$ is defined as follows.  Let $A \in
\bgsalgu$. Using the universal property of $C^{*}(A)$, we
obtain unique $*$-homomorphisms $\delta_{C^{*}(A)} \colon
C^{*}(A) \to \cg \maxtimes C^{*}(A)$ and $\bar
\delta_{C^{*}(A)} \colon C^{*}(A) \to C^{*}(A) \maxtimes
\cg$ such that $\delta_{C^{*}(A)}(a) = \gamma \otimes a$ and
$\bar \delta_{C^{*}(A)}(a) = a\otimes \gamma'$ for all $a\in
A_{\gamma,\gamma'},\gamma,\gamma' \in A$, and with respect
to these $*$-homomorphisms, $C^{*}(A)$ becomes a
$(\complex,\Gamma)$-$C^{*}$-algebra.  Moreover, every
morphism $\pi\colon A\to C$ in $\cgsalgu$ extends uniquely
to a $*$-homomorphism $C^{*}(\pi) \colon C^{*}(A) \to
C^{*}(C)$ which is a morphism in $\cgcalg$.

 The functor $\bfU$ is defined as follows.
Let $A$ be a $(\complex,\Gamma)$-$C^{*}$-algebra and let
\begin{align*}
A_{\gamma,\gamma'}:= \{ a \in A : \delta(a) = \gamma
\otimes a, \bar\delta(a) = a\otimes \gamma'\} \subseteq A
\quad \text{for all } \gamma,\gamma' \in \Gamma.
\end{align*}
Then the sum $A_{*,*}:=\sum_{\gamma,\gamma'}
A_{\gamma,\gamma'} \subseteq A$ is a
$(\complex,\Gamma)$-$*$-algebra, and, every morphism
$\pi\colon A \to \complex$ of
$(\complex,\Gamma)$-$C^{*}$-algebras restricts to a morphism
$\pi_{*,*} \colon A_{*,*} \to C_{*,*}$ of
$(\complex,\Gamma)$-$*$-algebras.   We thus obtain a functor
$\bfU \colon \cgcalg \to \cgsalg$. 
\begin{lemma} \label{lemma:c-enveloping}
  $\bfU$ takes values in $\cgsalgu$.
\end{lemma}
\begin{proof}
  Let $A$ be a $(\complex,\Gamma)$-$C^{*}$-algebra. Then for
  every $*$-representation $\pi$ of $A_{*,*}$, the
  restriction to the $C^{*}$-subalgebra $A_{e,e}$ is
  contractive and thus $\|\pi(a)\|^{2}=\|\pi(a^{*}a)\| \leq
  \|a^{*}a\|=\|a\|^{2}$ for all $a\in A_{\gamma,\gamma'}$
  $\gamma,\gamma'\in \Gamma$. Since such elements $a$ span
  $A_{*,*}$, we can conclude $|a'|<\infty$ for all $a'\in
  A_{*,*}$.
\end{proof}
For every $(\complex,\Gamma)$-$C^{*}$-algebra
$A$, the morphisms $p_{A}$ and
$\bar p_{A}$ yield a morphism
\begin{align*}
  P_{A} &:= \bar p_{A} \circ (p_{A} \maxtimes \Id) = p_{A}
  \circ (\Id \maxtimes \bar p_{A}) \colon \cg \maxtimes A
  \maxtimes \cg \to A && \text{in} && \cgcalg^{\cp}.
\end{align*}
\begin{lemma} \label{lemma:c-PA}
  \begin{enumerate}
  \item  Let $A \in \cgcalg$. Then for all 
    $\gamma,\gamma',\beta,\beta' \in \Gamma$,
   \begin{align*}
     P_{A}(\gamma \otimes A\otimes \gamma')
     &=A_{\gamma,\gamma'}, & P_{A}(\beta \otimes
     A_{\gamma,\gamma'} \otimes \beta') &=
     \delta_{\beta,\gamma}\delta_{\beta',\gamma'}A_{\gamma,\gamma'},
     & A&= \overline{A_{*,*}}.
   \end{align*} 
 \item Let $A \in \cgsalgu$. Then $
   C^{*}(A)_{\gamma,\gamma'} =
   \overline{A_{\gamma,\gamma'}}$ for all $\gamma,\gamma'\in
   \Gamma$.
  \end{enumerate}
\end{lemma}
\begin{proof}
  We only prove i); assertion ii) follows similarly. First,
  $P_{A}(\gamma \otimes A \otimes \gamma') \subseteq
  A_{\gamma,\gamma'}$ because $P_{A}$ is a morphism in
  $\cgcalg$ and $\Delta_{\Gamma}(\gamma'')=\gamma''\otimes
  \gamma''$ for $\gamma''=\gamma,\gamma'$ .  This inclusion,
  the relation $C^{*}(\Gamma) \maxtimes A\maxtimes
  C^{*}(\Gamma) = \overline{\sum_{\gamma,\gamma'}
    \gamma\otimes A\otimes \gamma' }$ and continuity and
  surjectivity of $P_{A}$ imply $\overline{A_{*,*}}=A$. The
  equation $P_{A}(\beta \otimes A_{\gamma,\gamma'} \otimes
  \beta') =
  \delta_{\beta,\gamma}\delta_{\beta',\gamma'}A_{\gamma,\gamma'}$
  follows from the definitions and implies that the
  inclusion $P_{A}(\gamma \otimes A \otimes \gamma')
  \subseteq A_{\gamma,\gamma'}$ is an equality.
\end{proof}

 For every $A$ in $\cgsalgu$ and $C$ in $\cgcalg$,
we get canonical morphisms $\eta_{A} \colon A \to
C^{*}(A)_{*,*}$ in $\cgsalgu$ and $\epsilon_{C} \colon
C^{*}(C_{*,*}) \to C$ in $\cgcalg$.
\begin{proposition} \label{proposition:c-adjoints}
  The functors $\bfC$ and $\bfU$ are adjoint, where the unit
  and counit of the adjunction are the families
  $(\eta_{A})_{A}$ and $(\epsilon_{C})_{C}$,
  respectively. Furthermore, $\bfU$ is faithful.
\end{proposition}
\begin{proof}
  Let $A \in \bgsalgu$ and $C\in \bgcalg$. Since the
  representation $A \to C^{*}(A)$ has dense image and is
  universal, the assignment $(C^{*}(A) \xrightarrow{\pi} C)
  \mapsto (A \xrightarrow{\eta_{A}} C^{*}(A)_{*,*}
  \xrightarrow{\pi_{*,*}} C_{*,*})$ yields a bijective
  correspondence between morphisms $C^{*}(A)\to C$ and
  morphisms $A \to C_{*,*}$.    The functor $\bfU$   is
  faithful because  $A_{*,*} \subseteq A$ is dense.
\end{proof}
\begin{remark} \label{remark:c-monoidal}
 Similar arguments as in the proof of Lemma
\ref{lemma:c-PA} show that  
for all $A,C \in \cgcalg$, $D,E \in \cgsalgu$ and
  all $\gamma,\gamma'' \in \Gamma$,
 \begin{align*}
   (A \gmaxtimes C)_{\gamma,\gamma''} &=
   \overline{\sum_{\gamma'} A_{\gamma,\gamma'} \otimes
     C_{\gamma',\gamma''}}, & (C^{*}(D) \gmaxtimes
   C^{*}(E))_{\gamma,\gamma'} &= \overline{\sum_{\gamma'}
     D_{\gamma,\gamma'} \otimes
     E_{\gamma',\gamma''}}.  \end{align*}
\end{remark}
A \emph{short exact sequence} of
$(\complex,\Gamma)$-$C^{*}$-algebra is a sequence of
morphisms $J \xrightarrow{\iota} A \xrightarrow{\pi} C$ in
$\cgcalg$ such that $\ker\iota=0$, $\iota(J)=\ker \pi$ and $\pi(A)=C$. A functor
on $\cgcalg$ is \emph{exact} if it maps  short exact
sequences to short exact sequences.
\begin{proposition} \label{proposition:c-exact}
  For every $(\complex,\Gamma)$-$C^{*}$-algebra $D$, the
  functors $-\gmaxtimes D$ and $D\gmaxtimes -$ on $\cgcalg$
  are exact.
\end{proposition}
\begin{proof}
  If $J \xrightarrow{\iota} A \xrightarrow{\pi} C$ is a
  short exact sequence in $\cgcalg$, then $J \maxtimes D
  \xrightarrow{\iota \maxtimes \Id} A\maxtimes D
  \xrightarrow{\pi\maxtimes \Id} C \maxtimes D$ is exact
  (see, e.g., \cite[Proposition 3.7]{brown-ozawa}), whence
  $\ker (\iota \gmaxtimes \Id) = 0$ and
  \begin{align*}
    \ker(\pi \gmaxtimes \Id) = p_{A,D}(\ker (\pi \maxtimes
    \Id)) &= p_{A,D}((\iota \maxtimes \Id)(J\maxtimes D)) \\ &=
    (\iota \gmaxtimes \Id)(p_{J,D}(J\maxtimes D)) =     (\iota
    \gmaxtimes \Id)(J\gmaxtimes D), \\
    (\pi \gmaxtimes \Id)(A \gmaxtimes D) &= (\pi \gmaxtimes
    \Id)(p_{A, D}(A \maxtimes D)) \\ &= p_{C,D} ((\pi \maxtimes
    \Id)(A\maxtimes D)) = p_{C,D}(C\maxtimes D) = C\gmaxtimes
    D. \qedhere
  \end{align*}
\end{proof}
\subsection{The monoidal category of $(B,\Gamma)$-$C^{*}$-algebras}

\label{section:cb}

We now define an analogue of $(B,\Gamma)$-$*$-algebras on
the level of universal $C^{*}$-algebras, and construct a
monoidal product  which is unital and associative.
\begin{definition} \label{definition:c-bg-algebra} A
  \emph{$(B,\Gamma)$-$C^{*}$-algebra} is a
  $(\complex,\Gamma)$-$C^{*}$-algebra $A$ equipped with
  unital $*$-homomorphisms $r_{A}, s_{A}\colon B \to
  A_{e,e}$ such that $A_{*,*}$ is a $(B,\Gamma)$-$*$-algebra
  with respect to the map $r_{A}\times s_{A} \colon B\otimes
  B \to A_{e,e}$, $b\otimes b' \mapsto r_{A}(b)s_{A}(b')$.
  A \emph{morphism} of $(B,\Gamma)$-$C^{*}$-algebras is a
  $B\otimes B$-linear morphism of
  $(\complex,\Gamma)$-$C^{*}$-algebras.  We denote by
  $\bgcalg$ the category of all
  $(B,\Gamma)$-$C^{*}$-algebras.  Replacing
  $*$-homomorphisms by c.p.-maps, we define c.p.-maps of
  $(B,\Gamma)$-$C^{*}$-algebras and the category
  $\bgcalg^{\cp}$.
\end{definition}

Denote by $\bgsalgu \subseteq \bgsalg$ the full subcategory
formed by all $(B,\Gamma)$-$*$-algebras that have an
enveloping $C^{*}$-algebra. This category is related to
$\bgcalg$ as follows.  If $C\in \bgcalg$, then $C_{*,*} \in
\bgsalgu$ by Lemma \ref{lemma:c-enveloping}. Conversely, if
$A \in \bgsalgu$, then $C^{*}(A)$ carries a natural
structure of a $(B,\Gamma)$-$C^{*}$-algebra.  The
canonical maps $\eta_{A} \colon A \to C^{*}(A)_{*,*}$ and
$\epsilon_{C} \colon C^{*}(C_{*,*}) \to C$ are morphisms in
$\bgsalgu$ and $\bgcalg$, respectively, and  Proposition
\ref{proposition:c-adjoints} therefore implies:
\begin{corollary} \label{corollary:c-bg-adjoints}
  The assignments $A\mapsto C^{*}(A), \pi \mapsto
  C^{*}(\pi)$ and $A \mapsto A_{*,*}, \pi \mapsto \pi_{*,*}$
  form a pair of adjoint functors
  \begin{align*}
    \xymatrix@C=40pt{ \bgsalgu \ar@<+3pt>[r]^{\bfC} &
      \bgcalg \ar@<+3pt>[l]^{\bfU,}  }
  \end{align*}
  where the unit and counit of the adjunction are the
  families $(\eta_{A})_{A}$ and $(\epsilon_{C})_{C}$,
  respectively. Furthermore, $\bfU$ is faithful.
\end{corollary}

Let $A$ and $C$ be $(B,\Gamma)$-$C^{*}$-algebras. Then the
$(\complex,\Gamma)$-$C^{*}$-algebra $A \gmaxtimes C$ is a 
$(B,\Gamma)$-$C^{*}$-algebra with respect to the $*$-homomorphisms $r \colon b\mapsto r_{A}(b) \maxtimes 1$
and $s \colon b' \mapsto 1
\maxtimes s_{C}(b')$, and the assignments $(A,C) \mapsto
A\gmaxtimes C$ and $(\phi,\psi)\mapsto \phi\gmaxtimes \psi$
define a product $-\gmaxtimes-$ on $\bgcalg^{(\cp)}$ that is
associative in the obvious sense. Using the map
\begin{align*}
  t_{A,C} \colon B \to A\gmaxtimes C, \quad b\mapsto
  s_{A}(b)\maxtimes 1 - 1\maxtimes r_{C}(b),
\end{align*}
we define an ideal $(t_{A,C}(B)) \subseteq A\gmaxtimes
C$. Since $t_{A,C}(B) \subseteq (A \gmaxtimes C)_{e,e}$, the quotient
\begin{align*}
  A \gbmaxtimes C &:= (A \gmaxtimes C)/(t_{A,C}(B)).
\end{align*}
inherits the $(B,\Gamma)$-$C^{*}$-algebra structure of
$A\gmaxtimes C$. For every pair of morphisms $\phi \colon
A\to C$ and $\psi \colon D\to E$ in $\bgcalg^{(\cp)}$, the
morphism $\phi \gmaxtimes \psi$ maps $t_{A,D}(B)$ to
$t_{C,E}(B)$ and thus factorizes to a morphism $\phi
\gbmaxtimes \psi \colon A\gbmaxtimes D \to C\gbmaxtimes
E$. We thus obtain a product $-\gbmaxtimes-$ on
$\bgcalg^{(\cp)}$, and the canonical quotient map
$q_{A,C} \colon A \gmaxtimes C \to A\gbmaxtimes C$ yields a
natural transformation $q=(q_{A,C})_{A,C}$ from
$-\gmaxtimes-$ to $-\gbmaxtimes-$.
\begin{remarks}
  \begin{enumerate}
  \item For all $(B,\Gamma)$-$C^{*}$-algebras $A,C$, we have
    $ [t_{A,C}(B) (A\gmaxtimes C)] = (t_{A,C}(B)) =
    [(A\gmaxtimes C)t_{A,C}(B)]$.  Indeed, a short
    calculation shows that for all $\gamma,\gamma',\gamma''
    \in \Gamma$, $a \in A_{\gamma,\gamma'}$, $c\in
    C_{\gamma',\gamma''}$, $b\in B$, $(a \otimes
    c)t_{A,C}(b) = t_{A,C}(\gamma'(b))(a\otimes c)$, and now
    the assertion follows from Remark
    \ref{remark:c-monoidal}.
  \item For every $(B,\Gamma)$-$C^{*}$-algebra $D$, the
    functors $-\gbmaxtimes D$ and $D\gbmaxtimes -$ on
    $\bgcalg$ preserve surjections because the functors
    $-\gmaxtimes D$ and $D\gmaxtimes -$ do so by Proposition
    \ref{proposition:c-exact}.
  \end{enumerate}
\end{remarks}

We show that the full crossed product $B\hat\rtimes
\Gamma:=C^{*}(B\rtimes \Gamma)$ is the unit for the product
$-\gbmaxtimes-$.  Denote by $\iota_{\Gamma} \colon
C^{*}(\Gamma) \to B\hat\rtimes \Gamma$ the natural
inclusion.
\begin{proposition}
  \begin{enumerate}
  \item For each $(B,\Gamma)$-$C^{*}$-algebra $A$, the
    $*$-homomorphisms
    \begin{align*}
      L_{A} &\colon A \xrightarrow{\delta_{A}} \cg
      \gmaxtimes A \xrightarrow{\iota_{\Gamma} \gmaxtimes
        \Id} (B\hat\rtimes \Gamma) \gmaxtimes A
      \xrightarrow{q_{B\hat\rtimes \Gamma,A}} (B\hat\rtimes
      \Gamma) \gbmaxtimes A
\end{align*}
and
\begin{align*}
      R_{A} &\colon A \xrightarrow{\bar \delta_{A}}
      A\gmaxtimes \cg \xrightarrow{\Id\gmaxtimes
        \iota_{\Gamma}} A \gmaxtimes (B\hat\rtimes \Gamma)
      \xrightarrow{q_{A,B\hat\rtimes \Gamma}} A \gbmaxtimes
      (B\hat\rtimes \Gamma),
    \end{align*}
    are isomorphisms of $(B,\Gamma)$-$C^{*}$-algebras.
  \item The families $R=(R_{A})_{A}$ and $L=(L_{A})_{A}$
    form natural isomorphism from $\Id$ to $((B\hat \rtimes
    \Gamma)\gbmaxtimes -)$ and $(-\gbmaxtimes (B\hat \rtimes
    \Gamma))$, respectively, regarded as functors on
    $\bgcalg^{(\cp)}$.
  \end{enumerate}
\end{proposition}
\begin{proof}
  One easily checks that each $L_{A}$ is a morphism of
  $(B,\Gamma)$-$C^{*}$-algebras and that $L=(L_{A})_{A}$ is
  a natural transformation. We show that $L_{A}$ is an
  isomorphism for every $(B,\Gamma)$-$C^{*}$-algebra
  $A$. The assertions concerning $R=(R_{A})_{A}$ then follow
  similarly.

  To prove that $L_{A}$ is surjective, we only need to show
  that $(t_{B\hat\rtimes\Gamma,A}(B)) + C^{*}(\Gamma)
  \gmaxtimes A$ is dense in $(B\hat\rtimes \Gamma)
  \gmaxtimes A$. But by Remark \ref{remark:c-monoidal},
  elements of the form
  \begin{align*}
    b\gamma \otimes a &= t_{B\hat\rtimes \Gamma,A}(b)
    (\gamma \otimes a) + \gamma \otimes r_{A}(b)a, \quad
    \text{where } b\in B,a\in
    A_{\gamma,\gamma'},\gamma,\gamma' \in \Gamma,
  \end{align*}
  are linearly dense in $(B\hat\rtimes \Gamma) \gmaxtimes
  A$. 

  To prove that $L_{A}$ is injective, we only need to show
  that the intersection 
  \begin{align*}
    J:=(\iota_{\Gamma}(\cg) \gmaxtimes A) \cap
    (t_{B\hat\rtimes \Gamma,A}(B)) \subseteq (B\hat\rtimes
    \Gamma) \gmaxtimes A
  \end{align*}
  equals $0$. Since $J=\overline{J_{*,*}}$ by Lemma \ref{lemma:c-PA},
  it suffices to show that $J_{\gamma,\gamma'} = 0$ for all
  $\gamma,\gamma'\in \Gamma$. Note that $J_{\gamma,\gamma'}
  = [\gamma \otimes A_{\gamma,\gamma'}] \cap [(B\gamma
  \otimes A_{\gamma,\gamma'})t_{B\hat\rtimes
    \Gamma,A}(B)]$. For each $\gamma,\gamma' \in \Gamma$,
  define a linear map $R_{\gamma,\gamma'}\colon B\gamma
  \otimes A_{\gamma,\gamma'} \to A_{\gamma,\gamma'}$ by
  $b\gamma \otimes a \mapsto r(b)a$. Then $R_{e,e}$ extends
  to a $*$-homomorphism on the $C^{*}$-subalgebra $B
  \maxtimes A_{e,e} \subseteq (B\hat\rtimes \Gamma)
  \gmaxtimes A$, and each $R_{\gamma,\gamma'}$ extends to a
  bounded linear map on $[B\gamma \otimes
  A_{\gamma,\gamma'}] \subseteq (B\hat\rtimes \Gamma)
  \gmaxtimes A$ because
  \begin{align*}
    \|R_{\gamma,\gamma'}(z)\|^{2} =
    \|R_{\gamma,\gamma'}(z)R_{\gamma,\gamma'}(z)^{*}\| =
    \|R_{e,e}(zz^{*})\| \leq \|zz^{*}\| = \|z\|^{2}
  \end{align*} 
  for all $z \in B\gamma \otimes A_{\gamma,\gamma'}$.  Now,
  $R_{\gamma,\gamma'}(zt_{B\hat\rtimes\Gamma,A}(b)) = 0$ for
  all $z \in [B\gamma \otimes A_{\gamma,\gamma'}]$ and $b\in
  B$, and $R_{\gamma,\gamma'}(\gamma \otimes a) = a$ for all
  $a \in A_{\gamma,\gamma'}$. Consequently,
  $J_{\gamma,\gamma'} =0$.
\end{proof}

We now show that the product $-\gbmaxtimes-$ is
associative.  Let $A,C,D$ be $(B,\Gamma)$-$C^{*}$-algebras,
denote by $a_{A,C,D} \colon (A \gmaxtimes C) \gmaxtimes D
\to A \gmaxtimes (C\gmaxtimes D)$ the canonical isomorphism
and let
\begin{align*}
  \Phi_{A,C,D} := q_{\big(A\gbmaxtimes C\big),D} \circ
  (q_{A,C} \gmaxtimes \Id) &\colon (A \gmaxtimes C)
  \gmaxtimes D \to (A
  \gbmaxtimes C) \gbmaxtimes D, \\
  \Psi_{A,C,D} := q_{A,\big(C\gbmaxtimes D\big)} \circ
  (\Id \gmaxtimes q_{C,D}) &\colon A \gmaxtimes
  (C\gmaxtimes D) \to A \gbmaxtimes (C\gbmaxtimes D).
\end{align*}
\begin{lemma} \label{lemma:c-associativity}
  \begin{enumerate}
  \item $\ker \Phi_{A,C,D}$ and $ \ker \Psi_{A,C,D}$ are
    generated as ideals by
    $t_{A,C}(B) \otimes 1_{D} + t_{\big(A\gmaxtimes
        C,D\big)}(B)$ and $1_{A} \otimes
      t_{C,D}(B) + t_{A,\big(C\gmaxtimes D\big)}(B)$,
    respectively.
  \item There exists a unique isomorphism of
    $(B,\Gamma)$-$C^{*}$-algebras $\tilde a_{A,C,D} \colon
    (A\gbmaxtimes C) \gbmaxtimes D \to A \gbmaxtimes
    (C\gbmaxtimes D)$ such that $\tilde a_{A,C,D} \circ
    \Phi_{A,C,D} =\Psi_{A,C,D} \circ a_{A,C,D}$.
  \end{enumerate}
\end{lemma}
\begin{proof}
  i) By Proposition \ref{proposition:c-exact}, $\ker (q_{A,C}
  \gmaxtimes \Id_{D}) = (\ker q_{A,C}) \gmaxtimes D =
  (t_{A,C}(B)) \gmaxtimes D$, and $\ker q_{\big(A\gbmaxtimes
    C\big),D}$ is generated as an ideal by $(q_{A,C}
  \gmaxtimes \Id_{D})\big(t_{\big(A\gmaxtimes
    C,D\big)}(B)\big)$. The assertion on $\Phi_{A,C,D}$
  follows, and the assertion concerning $\Psi_{A,C,D}$
  follows similarly.

  ii) Using i), one easily verifies that $a_{A,C,D}(\ker
  \Phi_{A,C,D}) =\ker \Psi_{A,C,D}$. We thus get an
  isomorphism $\tilde a_{A,C,D}$ of $C^{*}$-algebras which
  is easily seen to be an isomorphism of
  $(B,\Gamma)$-$C^{*}$-algebras.
\end{proof}
\begin{proposition}
  The family $(\tilde a_{A,C,D})_{A,C,D}$ is a natural
  isomorphism from $(-\gbmaxtimes-)\gbmaxtimes -$ to
  $-\gbmaxtimes(-\gbmaxtimes-)$.
\end{proposition}
\begin{proof}
  By Lemma \ref{lemma:c-associativity}, we only need to
  check naturality which is  straightforward.
\end{proof}

\subsection{Free dynamical quantum groups on the level of
  universal $C^{*}$-algebras}
\label{section:free-c}

Given the monoidal structure on the category of all
$(B,\Gamma)$-$C^{*}$-algebras, the definitions in
\S\ref{section:bg}--\S\ref{section:ao} carry over  as follows:
\begin{definition}
  A \emph{compact $(B,\Gamma)$-Hopf $C^{*}$-algebroid} is a
  $(B,\Gamma)$-$C^{*}$-algebra $A$ with a morphism $\Delta
  \colon A \to A \gbmaxtimes A$ satisfying
  \begin{enumerate}
  \item $(\Delta \gbmaxtimes \Id)\circ\Delta = (\Id
  \gbmaxtimes \Delta)\circ\Delta$ (\emph{coassociativity}),
\item $[\Delta(A)(1 \otimes A_{e,*})]=A \gbmaxtimes A =
  [(A_{*,e} \otimes 1)\Delta(A)]$, where $A_{e,*} =
  [\sum_{\gamma} A_{e,\gamma}] \subseteq A$ and
  $A_{*,e}=[\sum_{\gamma} A_{\gamma,e}] \subseteq A$
  (\emph{cancellation}).
\end{enumerate}
A \emph{counit} for a compact $(B,\Gamma)$-Hopf
$C^{*}$-algebroid $(A,\Delta)$ is a morphism $\epsilon \colon
A \to B\hat\rtimes \Gamma$  of $(B,\Gamma)$-$C^{*}$-algebras
satisfying $(\epsilon \gbmaxtimes \Id) \circ \Delta =
\Id_{A} = (\Id \gbmaxtimes \epsilon) \circ \Delta$. A
\emph{morphism} of compact $(B,\Gamma)$-Hopf
$C^{*}$-algebroids  $(A,\Delta_{A})$ and $(C,\Delta_{C})$ is a morphism $\pi\colon A\to C$ satisfying $\Delta_{C}\circ
\pi = (\pi\gbmaxtimes\pi)\circ \Delta_{A}$. We denote the
category of all compact $(B,\Gamma)$-Hopf $C^{*}$-algebroids
by $\cHopf_{(B,\Gamma)}$.
\end{definition}
Denote by $\Hopf_{(B,\Gamma)}^{0}$ the full subcategory of
$\Hopf_{(B,\Gamma)}^{*}$ formed by all $(B,\Gamma)$-Hopf
$*$-algebroids $(A,\Delta,\epsilon,S)$ where $A \in
\bgsalgu$. 
\begin{proposition}
  Let $(A,\Delta,\epsilon,S)
  \in\Hopf_{(B,\Gamma)}^{*}$. Then $\Delta$ extends to a
  morphism of $(B,\Gamma)$-$C^{*}$-algebras
  $\Delta_{C^{*}(A)} \colon C^{*}(A) \to C^{*}(A)
  \gbmaxtimes C^{*}(A)$ such that
  $(C^{*}(A),\Delta_{C^{*}(A)})$ is a compact
  $(B,\Gamma)$-Hopf $C^{*}$-algebroid with counit
  $C^{*}(\epsilon) \colon C^{*}(A) \to C^{*}(B\rtimes
  \Gamma) = B\hat\rtimes \Gamma$.
\end{proposition}
\begin{proof}
  The composition of $\Delta$ with the canonical map
  $A\todot A \to C^{*}(A) \gbmaxtimes C^{*}(A)$ extends to a
  morphism $\Delta_{C^{*}(A)}$ by the universal property of
  $C^{*}(A)$.  Coassociativity of $\Delta$ and density of
  $A$ in $C^{*}(A)$ imply coassociativity of
  $\Delta_{C^{*}(A)}$, and cancellation follows from Remark
  \ref{remarks:bg-hopf} ii).
\end{proof}
The assignments $(A,\Delta,\epsilon,S) \mapsto
(C^{*}(A),\Delta_{C^{*}(A)})$ and $\pi \mapsto C^{*}(\pi)$
evidently form a functor $\Hopf_{(B,\Gamma)}^{0} \to
\cHopf_{(B,\Gamma)}$. 

We now apply this functor to the free unitary and free
orthogonal dynamical quantum groups $\Au(\nabla,F)$ and
$\Ao(\nabla,F,G)$ introduced in Definition
\ref{definition:intro-au-hopf}, Theorem
\ref{theorem:intro-au-hopf} and Definition
\ref{definition:matrix-hopf-involution}, Theorem
\ref{theorem:ao-prime-hopf}, respectively. 

Let $\gamma_{1},\ldots,\gamma_{n} \in \Gamma$ and $\nabla =
\diag(\gamma_{1},\ldots,\gamma_{n}) \in \Mn(B\rtimes
\Gamma)$.

Assume that $F \in \GLn(B)$ be $\nabla$-even in the sense
that $\nabla F\nabla^{-1} \in \Mn(B)$. Then   the
  $(B,\Gamma)$-Hopf $*$-algebroid $\Au(\nabla,F)$  is generated by a copy of
  $B\otimes B$ and entries of a unitary matrix $v\in
  \Mn(\Au(\nabla,F))$ and therefore has an enveloping
  $C^{*}$-algebra.  Applying the functor $\bfC$ and
  unraveling the definitions, we find:
  \begin{corollary}
    $C^{*}(\Au(\nabla,F))$ is the universal $C^{*}$-algebra
    generated by a inclusion $r\times s$ of $B\otimes B$ and
    by the entries of a unitary $n\times n$-matrix $v$
    subject to the relations 
    \begin{enumerate}
    \item $v_{ij}r(b)=r(\gamma_{i}(b))v_{ij}$ and
      $v_{ij}s(b)=s(\gamma_{j}(b))v_{ij}$ for all $i,j$ and
      $b\in B$,
    \item  $v^{-\top} = \bar v$ is
    invertible and $r_{n}(\nabla F\nabla^{-1})\bar
    v^{-\top}=vs_{n}(F)$.
    \end{enumerate}
    It has the
    structure of a compact $(B,\Gamma)$-Hopf
    $C^{*}$-algebroid with counit, where for all $i,j$,
    \begin{align} \label{eq:explicit}
      \delta(v_{ij}) &= \gamma_{i} \otimes v_{ij}, &\bar
      \delta(v_{ij})  &= v_{ij} \otimes \gamma_{j}, &
      \Delta(v_{ij}) &= \sum_{k} v_{ik} \otimes v_{kj}, &
      \epsilon(v_{ij}) &= \delta_{i,j}.
    \end{align}
  \end{corollary}
  Let $F,G \in \GLn(B)$ be $\nabla$-odd in the sense that
  $\nabla F \nabla, \nabla G\nabla \in \Mn(B)$, and assume
  that $GF^{*}=FG^{*}$. If $G^{-1}F=\lambda \bar HH^{\top}$
  for some $\lambda \in \complex$ and some $\nabla$-even
  $H\in \GLn(B)$, then $\Ao(\nabla,F,G)$ is generated by a
  copy of $B\otimes B$ and entries of a unitary matrix
  $u\in \Mn(\Ao(\nabla,F,G))$ by Remark \ref{remarks:ao-fg}
  iii), and therefore has an enveloping $C^{*}$-algebra.
  \begin{corollary}
    $C^{*}(\Ao(\nabla,F,G))$ is is the universal
    $C^{*}$-algebra generated by a inclusion $r\times s$ of
    $B\otimes B$ and by the entries of an invertible $n\times
    n$-matrix $v$ subject to the relations 
    \begin{enumerate}
    \item $v_{ij}r(b)=r(\gamma_{i}(b))v_{ij}$ and
      $v_{ij}s(b)=s(\gamma_{j}(b))v_{ij}$ for all $i,j$ and
      $b\in B$,
    \item $r_{n}(\nabla F\nabla) v^{-\top}=vs_{n}(F)$ and $r_{n}(\nabla G\nabla)\bar v = vs_{n}(G)$.
    \end{enumerate}
    It carries the
    structure of a compact $(B,\Gamma)$-Hopf
    $C^{*}$-algebroid with counit such that
    \eqref{eq:explicit} holds.
  \end{corollary}

\medskip

\emph{Acknowledgments.} I thank Erik Koelink for introducing
me to dynamical quantum groups and for  stimulating
discussions, and the referee for helpful suggestions.

\def\cprime{$'$}

\end{document}